\documentclass[a4paper,12pt]{article}
\usepackage[left=3cm,right=3cm,top=3.5cm,bottom=3.2cm]{geometry}

\usepackage[T1]{fontenc}
\usepackage{frcursive}
\usepackage[latin1]{inputenc}
\usepackage{manfnt} 
\usepackage{amssymb,amsmath,amsthm, amscd}
\usepackage{color}
  \usepackage[colorlinks=true,linkcolor=black,citecolor=black,urlcolor=black]{hyperref}
 \usepackage[usenames,dvipsnames]{xcolor}
 \usepackage{bbm} 
 \usepackage{float}
\usepackage{comment}
\newsavebox{\strikeoutI}
\usepackage{latexsym}
\usepackage{graphicx}
\usepackage{enumerate}
\usepackage{color}
\newtheorem{Theorem}{Theorem}[part]

\newtheorem{Proposition}{Proposition}[part]

\newtheorem{Lemma}{Lemma}[part]

\newtheorem{Remark}{Remark}[part]

\newcommand{\bsde}{backward stochastic differential equation\,\,}
\newcommand{\fbsde}{forward-backward stochastic differential equation\,\,}

\newcommand {\eeqn}{\end {eqnarray}}
\newcommand {\beq}{\begin {eqnarray*}}
\newcommand {\eeq}{\end {eqnarray*}}
\newcommand{\norm}[1]{\left\Vert#1\right\Vert}
\newcommand{\abs}[1]{\left\vert#1\right\vert}

\newcommand{\R}{\mathbb R}

\newcommand{\N}{\mathbb{N}}

 \newcommand{\tiun}{\ensuremath{ {t_{i+1}}}}
 \newcommand{\ti}{\ensuremath{ {t_{i}}}}

\def \QUAD{\mathbb \qquad \qquad \qquad}

\def\calF{\mathcal{F}}

\def\R{\mathbb R} \def\N{\mathbb N}     \def\E{\mathbb E}

\def\L{\mathbb L}

\def\E{\mathbb{E}}

\def\R{\mathbb{R}}

 \usepackage{amssymb}

\def \Sup{\displaystyle\sup}

\def \Max{\displaystyle\max}

\begin{document}
\title{
 \textbf{A Regress-Later Algorithm  for Backward   Stochastic Differential Equations}  
 }
 
\vspace{3mm} 
\author{
\begin{tabular}{c}
  Kossi  Gnameho \footnote{\tiny{Maastricht University , Dept. of Quantitative Economics,   Email: \sf   k.gnameho@maastrichtuniversity.nl }}    
\quad
Mitja  Stadje \footnote{\tiny{University of Ulm,  Faculty of Mathematics and Economics, Email:  \sf  mitja.stadje@uni-ulm.de} }  
   \quad     Antoon   Pelsser    \footnote{\tiny{Maastricht University,  Dept. of Quantitative Economics,  {Email}: \sf   a.pelsser@maastrichtuniversity.nl } } \\
  \vspace{1cm}
 \end{tabular}
 }
    \date{}
 
\maketitle

\begin{abstract} 
\noindent This work deals  with the numerical approximation of   backward stochastic differential equations (BSDEs).
 We propose   a new algorithm  which is based on the regression-later approach and   the least squares Monte Carlo method. 
 We give some conditions under which our numerical algorithm convergences and solve  two practical    experiments to illustrate its   performance. \\

\end{abstract}

\vspace{5mm}

\newpage
\tableofcontents
\newpage


\section{Introduction}
 
This work deals with the numerical approximation  of  backward stochastic differential equations (BSDEs) on a certain time interval $[0,T]$.   
Backward stochastic differential equations  were first introduced by Bismut \cite{bismut1973} in the linear case  and later developed  by  Pardoux and Peng \cite{Pardoux1992Peng}. 
  In the past decade,  BSDEs have attracted a lot of attention and have been intensively studied in mathematical finance, insurance and  stochastic optimal control theory. 
     For example   in   a complete financial market, the price of a standard European option can be seen as the solution of a linear BSDE. Moreover, the price of an American option can be formulated as the solution  of a reflected BSDE. 

\vspace{0.5cm}     
\noindent     BSDEs have also been widely applied for portfolio optimization,  indifference pricing,  modelling of convex risk measures  and the modelling of ambiguity with respect to the stochastic drift and the volatility.  
  See for instance  El Karoui et al. 
 \cite{ElKaroui1997BSDEinFinance}, Cheridito et al.  
 \cite{cheridito2007second}, Barles et al.  \cite{barles1997sde}, Duffie et al.  \cite{duffie1992stochastic}, Hamad\`ene et al. \cite{hamadene2007starting}, Hu et al. \cite{hu2005utility},  Leaven and Stadje \cite{laeven2014robust} and the references therein.
 In general, many of these equations do not have   an explicit or  {closed form solution}.    Due to its importance, some efforts have  been made to provide   numerical solutions. A  four-step scheme  has been proposed for instance by  Ma  et al. in \cite{jinma1994solving} to solve  forward-backward stochastic differential equations (FBSDEs). 
 In \cite{bally1997approximation},  Bally  has proposed  a random time discretization scheme.  
   Discrete time approximation  schemes have also been  proposed by  Bouchard and Touzi  in \cite{bouchard:touzi2004MC-BSDE} and   Chevance   \cite{chevance1997resolution} for instance. 
 In  Chevance's work \cite{chevance1997resolution}, strong regularity assumptions of the coefficients of the BSDE are requiered for the  convergence results.  In Crisan et al. \cite{crisan2012solving}, was proposed a cubature techniques   for BSDEs with application to nonlinear pricing.
  Gobet et al. \cite{gobet2005reg-MC-BSDE} presented a discrete  algorithm    based on the Monte Carlo  method to solve   BSDEs.
 Recently,  Fourier methods to solve FBSDEs were proposed by  Huijskens et al.  \cite{huijskens2016efficient} and a convolution method by Hyndman et al. \cite{hyndman2014global}.  Briand et al. \cite{briand2014simulation}  proposed an  algorithm to solve BSDEs based on Wiener chaos and Picard's iterations
expansion.    Gobet et al. \cite{gobet2016approximation}    designed a numerical scheme for solving BSDEs with   using Malliavin weights and least-squares regression.  Reducing variance in  numerical solution of BSDEs is proposed by   Alanko et al.  \cite{alanko2013reducing}.  
Other recent references can  be found  in  Chassagneux et al.   \cite{chassagneux2016numerical},  Khedher et al. \cite{khedher2016discretisation},  Bender et al.  \cite{bender2013posteriori},  Zhao et al. \cite{zhao2006new},  Ventura et al. \cite{ventura2016discretely},  Gong et al.  \cite{gong2015one},  among others. 
 
 \vspace{0.5cm}     
\noindent    We propose in this paper  a new algorithm  which is based on a regression-later approach. 
 Glasserman and Ye \cite{glasserman2004simulation}  show that the regression-later approach offers advantages  in comparison to  the classical regression  technique.   
An asymptotic convergence rate of the regression-later technique  is derived under  some  mild auxiliary   assumption  in  Beutner et  al.  \cite{beutner2013fast} for  single-period problems.
 Stentohoff \cite{stentoft2004convergence} discusses  the  convergence of the regression-now (cf. \cite{glasserman2004simulation}) technique in the case of the evaluation of  an American option.

\vspace{0.5cm}     
\noindent Under some regularity assumptions, the solution of a   FBSDE can be represented by the solution
 of a regular semi-linear parabolic partial differential equation (PDE). By
exploiting the Markov property of the solution of the FBSDE, we have developed a probabilistic
numerical regression called  \emph{regression-later algorithm}  based on the least squares
Monte Carlo method and the previous connection between the  quasi-linear parabolic partial differential equation  and the FBSDE. 
 To the  best of our knowledge, the regression-later approach has not been already used  in the BSDE   literature.   
For  most numerical algorithms to solve BSDEs, we need to compute  in general two conditional expectations      at each step   across the time interval. 
  This computation can be very costly especially in high dimensional problems.
For most numerical algorithms, it is important to note that  the process $Z$ is more difficult to compute   than the   process $Y$.
The proposed  algorithm requires only one conditional expectation to compute  at each step across the time interval.    The   algorithm   yields  good convergence results   in practice and the computation of $Z$ is simple. 
  
  \vspace{0.5cm}     
\noindent
This paper is structured  as follows. In the first part of our work, we introduce  the basic   theory of BSDEs, give  some general results on the studies of  FBSDEs  and review the classical backward Euler-Maruyama scheme. 
 
 In the next step, we   describe in detail the  regression-later algorithm  and  derive a  convergence result of the scheme.  
  Finally, we   provide two numerical experiments to illustrate the performance of the regression-later  algorithm: the first in the context of option pricing   and  the second discusses the case  where  the forward  process is a  Wiener process.

 \subsection*{Notations and Assumptions }
We will use  in this chapter  the notations of   El Karoui et al. \cite{el2008backward}.  We  consider  a filtered probability space $(\Omega,\mathcal{F},\mathbb{P},\mathbb{F})$   with  $\mathcal{F}= \mathcal{F}_{T}$,  $\mathbb{F}=(\mathcal{F}_{t})_{0\leq t\leq T}$ a complete natural filtration of a $d$-dimensional Brownian motion $W$   and $T$    a fixed finite horizon. For all $m \in \N^*$  and $x  \in \R^m $,  $|x|$  denotes the  Euclidean norm of $x$. 
   For the matrix $A \in  \R^{m\times d}$,  we define  its   Frobenius    norm by   $|A| : = \sqrt{\text{Trace}(AA^*) }$.  The matrix $A $ can be considered as an element of the space  $ \R ^{m\times d}$.  
\begin{itemize}

\item  $\L^2_m(\mathcal{F}_t):= \bigg \{(X_t)_{t \in [0,T]} \in \R^m, \, \mathcal{F}_t-\text{measurable and} $    
  $$\norm{X}_{\L^2}=\E[\abs{X_t}^2]^{1/2}<\infty \bigg\}.$$ 
 \item  $\mathcal{S}^2(\R^m):= \bigg\{(Y_t)_{t \in [0,T]} \in \R^m, \, \text{continuous and adapted such that} $ 
 $${ \displaystyle  \norm{Y}_{\mathcal{S}^2}^2=\E [ \sup_{t \in [0,T]} \abs{Y_t}^2] <\infty \bigg \}}  .$$
 \item  $\mathcal{H}^2(\R^m):= \bigg \{(Z_t)_{t \in [0,T]} \in \R^m ,\,\text{continuous and adapted such that}   $ 
 $$  { {\norm{Z}_{\mathcal{H}^2}^2=\E[(\int_0^T \abs{Z_s}^2ds) ]<\infty \bigg \}}}.$$ 
\item  All the equalities and the inequalities between random variables are understood in almost sure sense unless explicitly stated otherwise. 
 \item  $\mathcal{C}^{l,k}_b ([0,T]\times\R^m) $   is the set of real valued functions which are  $l$ times continuously differentiable  in their first coordinate and   $k$ times   in their second  coordinate with bounded partial derivatives up to order $k$.
 \item   $\mathcal{C}^k (\Omega) $   is the set of $k$ times continuously differentiable functions  on $\Omega$.
 
   \item For $x\in \R^m$,   $\nabla_x:= (\frac{\partial}{\partial x_1},  \cdots ,\frac{\partial}{\partial x_m})$. The operator  $\nabla_x$ is called the gradient.  In the one-dimensional case,  we will  use   the same  notation.
 \item For $x,y \in \R^m $,  $x.y$ denotes the usual inner product  on  the space  $\R^m$. 
 
\end{itemize} 
 
\section{Definitions and   Estimates}
\setcounter{equation}{0} \setcounter{Assumption}{0}
\setcounter{Theorem}{0} \setcounter{Proposition}{0}
\setcounter{Corollary}{0} \setcounter{Lemma}{0}
\setcounter{Definition}{0} \setcounter{Remark}{0}

 In this section, we introduce the general concept of backward stochastic differential equations  and forward-backward stochastic differential equations with respect to the standard Brownian motion. In the last part of the section, we recall some classical estimates from the theory of  BSDEs.
\subsection{Backward Stochastic Differential Equations}
In  the  filtered probability space $(\Omega,\mathcal{F},\mathbb{P},\mathbb{F})$,  backward stochastic differential equations  are a special class of stochastic differential equations. The main difference is that these equations are specified with  a prescribed terminal value as shown in the following   equation  
%
 \begin{equation}
  \left\{ 
  \label{art1_EDSR00}
     \begin{aligned}
   	    -&dY_t= f(t,Y_t,Z_t) dt - Z_tdW_t ; \qquad 0\leq t< T, \\
   	    &Y_T=\xi.
   	    \end{aligned}
  \right.
 \end{equation}
The preceding system can be written equivalently as the following   
\begin{equation}
 \label{art1_EDSR01}
  \left. 
     \begin{aligned}
   	    Y_t= \xi+ \int_t^T f(s,Y_s,Z_s) ds - \int_t^T Z_sdW_s,\\
   	    \end{aligned}
  \right.
 \end{equation}
where 
\begin{itemize}
\item  $\xi $ is the terminal condition of the equation \eqref{art1_EDSR00} and is assumed   to be an $\mathcal{F}_T$- measurable and a square integrable random variable, 
\item the measurable mapping $(t, y, z)\mapsto f(t, y, z)$ is generally called the generator of \eqref{art1_EDSR00}.    
 \end{itemize}

 \noindent A   solution of the \bsde  \eqref{art1_EDSR00} is   a couple of progressively measurable processes $(Y,Z)$  such that:  
  \begin{eqnarray*}
 \left\{ \begin{aligned}
  &i)  \quad \int_0^T|Z_s|^2ds < \infty \text{ and }  \int_0^T \abs{ f(s,Y_s, Z_s) }ds < \infty,       \\
 &ii)   \quad (Y_t ,Z_t) \quad \text{satisfies the equation}   ~\eqref{art1_EDSR00},     
 \end{aligned}
      \right.
 \end{eqnarray*}
%
%
In general, the equation  \eqref{art1_EDSR00}  does not  admit  a unique solution.
The existence and uniqueness of a solution can be shown under the conditions  given in    Pardoux and Peng \cite{pardoux1990adapted} which involves the Lipschitz continuity of the driver function $g$.  In that case, we have 
$$(Y_t ,Z_t)_{0\leq t\leq T}\in \mathcal{S}^2(\R^m) \times \mathcal{H}^2(\R^{m\times d}).$$
 
\begin{Remark}{}\label{art1_remark1}
  If  the generator function   is identically equal to zero, the \bsde \eqref{art1_EDSR01} is reduced to the following classical stochastic equation
  $$  Y_t=\xi  -  \int_t^TZ_sdW_s.$$  
 This previous simplification  can be associated with the martingale representation theorem  in the filtration generated by the Brownian motion. The  solution $Y$  is a martingale  and  we have the explicit solution
 $$  Y_t=\E(\xi | \mathcal{F}_t).$$  
 \end{Remark}
\noindent BSDEs appear  in numerous problems in finance, in insurance and especially in stochastic control. 
A frequent problem in finance or insurance is  the problem of the valuation of  contract and the  risk management of a portfolio.   Linear and nonlinear BSDEs  appear naturally in these situations. For example   in   a complete financial market, the price of a standard European option can be seen as the solution of a linear BSDE. 
The interested reader can consult  the paper of El Karaoui et  al.
  \cite{ElKaroui1997BSDEinFinance}, Cheridito et al.  
 \cite{cheridito2007second},  Duffie et al.  \cite{duffie1992stochastic}, Hamad\`ene et al. \cite{hamadene2007starting} and the references therein for further details.

\subsection{Forward-Backward Stochastic Differential Equations}
 
We will consider  decoupled forward-backward stochastic differential equations (FBSDEs), which consists of  a system of  two equations given by   
  \begin{eqnarray}
\label{art1_FBSDE}
 \left\{\begin{aligned}
 X^x_t &=x+\int_{0}^t b(s,X^x_s)  ds+ \int_{ 0}^t\sigma(s,X^x_s) dW_s, \quad (t,x) \in [0,T]\times \R^m,\\
 Y^x_t &=\phi(X^x_T) +\int_t^T f(s,X^x_s,Y^x_s,Z^x_s)ds-\int_t^T Z^x_s dW_s.
 \end{aligned} 
      \right.
 \end{eqnarray}
 The first component is a forward process and the second a backward  process.  In general, the system   \eqref{art1_FBSDE}  does not  admit  a unique solution.
The existence and uniqueness of a solution can be shown under the conditions  given in    Pardoux and Peng \cite{pardoux1990adapted} which involves the Lipschitz continuity property of the coefficient of 
the system \eqref{art1_FBSDE}.
  We make the following  regularity assumptions: 
  \begin{eqnarray*}
 {(H)}\left\{ 
 \begin{aligned}
 & (H1):\,\text{the functions}  \, (t,x) \mapsto b(t,x),\sigma(t,x)  \, \text{are   uniformly Lipschitz in $x$ } \\
&\quad \text{  and  satisfy: } \, | b(t,x)|+ | \sigma(t,x)| \leq K( 1+|x|),  \\
 &(H2):  \,\text{there exists a positive constant }\, K >0,\, \text{such that} \, \,     \\
 &   {|f(t_1,x_1,y_1,z_1)-f(t_2,x_2,y_2,z_2)|\leq K(|x_1-x_2|+ |y_1-y_2| + |z_1-z_2|)}  \\
  &\QUAD \text{for any} \quad    (t_i, x_i,y_i,z_i),  i=1,2, \\
  &(H3):\,\text{there exist}\, \, k_\sigma, K_\sigma  > 0\,\, \text{such that for all}\,\,   t \in [0,T ] , \, \,\text{and}\, \,  x,  \zeta  \in \R^m  \\
   &\QUAD  k_\sigma |\zeta|^2 \leq  \sum^{ }_{i,j}[ \sigma \sigma^ *] _{i,j}(t,x)\zeta_i \zeta_j | \leq K_\sigma |\zeta|^2, 
  \end{aligned}
 \right.
 \end{eqnarray*} 
 and 
 \begin{eqnarray*}
 {(G)}\left\{ 
 \begin{aligned}
  & (G1): \,\, \text{there exists a positive constant }\, K >0,\, \text{such that} \,  \\ 
   &\QUAD   \sup  |f(t,0,0,0)|\leq K, \,\,\, \text{for every}  \,\,\, t\in[0,T],  \\ 
 & (G2): \,\, \text{the function}  \, x \mapsto \phi(x)  \text{\, is Lipschitz and  belongs to } \mathcal{C}^1 (\R^m,\R)   \\ 
  &  \qquad \,\,  \text{almost everywhere and we denote  its Lipschitz constant  by $C_\phi$,}       \\ 
 &  {(G3):   \,\,\text{the driver} \, f: [0,T]\times\R^m\times\R\times\R^{d} \stackrel{}{\rightarrow}\, \R \, \,\,  \text{is continuously differentiable  }}  \\
   &  \qquad \,\,  \text{ in $(x,y,z)$ with uniformly bounded   derivatives,}  \\
 &    { (G4): \,\,  \text{the functions} \quad    b \in C^{0,1}_b([0,T]\times\R^m, \R^m) \,\, \text{and}\,\, \sigma \in C^{0,1}_b([0,T]\times\R^m,  \R^{m\times d} )}.\\ 
\end{aligned}
 \right.
 \end{eqnarray*}
\noindent  The  last assumption $(G4) $ means that  the functions $b$ and $\sigma$ are continuous  in their first coordinate and continuously differentiable  in the space variable with uniformly bounded derivatives. 
 The condition  $ \phi(X^x_T) \in  \L^2_m(\mathcal{F}_T)$ and the assumptions  $(H)$ and $(G1)$  ensure the existence and   uniqueness of the solution of the decoupled FBSDE \eqref{art1_FBSDE}.  With these assumptions, we have  
$$
(Y^x_t ,Z^x_t)_{0\leq t\leq T}\in \mathcal{S}^2(\R^m) \times \mathcal{H}^2(\R^{m\times d}).
$$
\noindent As already mentioned in the introduction, for a large class of FBSDEs,  we do not have an explicit solution. 
 We therefore need  approximation  schemes to solve these equations numerically.  
 Most of the existing numerical schemes are based on  the Monte Carlo  method.
Our regression algorithm is based on the   following theorem which  establishes a link between the solution of the  decoupled  FBSDE \eqref{art1_FBSDE} and the solution of the quasi-linear parabolic PDE \eqref{Art1_eq_PDE}. This theorem is  one of the cornerstones  for our  numerical scheme. 
\subsubsection*{Connection between Quasi-linear PDE and Forward-Backward SDE}
\begin{Theorem}(Pardoux and Peng \cite{Pardoux1992Peng}) \\
We assume that, there exist   $C>0$ and $q>0$, such that: 
  \begin{equation} \label{hypo_PDE}
  |u(t,x)| +  |(\nabla_x u)(t,x)| \leq C(1+|x|^q).
  \end{equation}
  The function $u \in C^{1,2}([0,T]\times\R^m)$   solves  the   parabolic partial differential equation  below: 
 \begin{eqnarray}\label{Art1_eq_PDE}
 \left\{ \begin{aligned} 
    & \frac{\partial u}{\partial t}(t,x)+\mathcal{L}u(t,x)+f(t,x,u(t,x), \nabla u \sigma(t,x))=0,    \quad (t,x)\in [0,T)\times\R^m \\
 & u(T,x)=\phi(x), \quad   \quad x\in \R^m.
 \end{aligned}
 \right.
 \end{eqnarray} 
The differential operator $\mathcal{L}$ is defined by  
\begin{equation*} 			        
 \quad \mathcal{L} \psi =:   b \nabla \psi  
 +  \frac{1}{2} Trace( A  \nabla ^2 \psi) ,  \quad \text{for any} \quad \psi \in C^{1,2}([0,T]\times\R^m),
 \end{equation*}
 and   {$A = \sigma\sigma^*$. The matrix $\sigma^*$ denotes the   transpose  matrix of  $\sigma$}.  
Then  the solution of the system \eqref{art1_FBSDE} can be  represented as follows: 
 \begin{equation*}
 \quad\forall  \,\, t\in[0,T],\quad  Y^x_t=u(t,X^x_t) \quad\mbox{\textup{and}}\quad Z^x_t=\sigma(t,X^x_t)^*\nabla_x u(t,X^x_t) .
  \end{equation*} 
%
 \end{Theorem}
\begin{proof}  
Let us first consider  a solution $u$ of the   parabolic partial differential equation \eqref{Art1_eq_PDE} and   the couple $(\bar{Y}_t, \bar{Z}_t) $ defined by   
$$\QUAD\bar{Y}_t=u(t,X^x_t),   \quad \bar{Z}_t=\sigma(t,X^x_t)^*\nabla_x u(t,X^x_t),  \quad  \text{for}  \, t \, \in[0,T]. $$   
It\^o's Lemma applied to the function $u$ leads us to   
 $$ u(t,X^x_t) = u(T,X^x_T)- \int_t^T\left(\frac{\partial u}{\partial s}(s,X^x_s)+\mathcal{L}u(s,X^x_s) \right) ds - \int_t^T\bar{Z}_s dW_s. $$ 
 \noindent As the function $ u $ solves the PDE \eqref{Art1_eq_PDE}, we have 
 $$ u(t,X^x_t) = u(T,X^x_t)+\int_t^T f( s,  X^x_s  , u(s,X^x_s), \nabla_x u(t,X^x_t) \sigma(t,X^x_t) )ds - \int_t^T\bar{Z}_s dW_s. $$
As $ u(T, X^x_T)= \phi(X^x_T)$,  one has   
$$ \bar{Y}_t  =\phi(X^x_T) +\int_t^T f(s,X^x_s,\bar{Y}_s,\bar{Z}_s)ds-\int_t^T \bar{Z}_s dW_s.$$
Hence the  couple  $(\bar{Y}_t, \bar{Z}_t)_ {t\in[0,T] }$ is a  solution of the \fbsde  in  \eqref{art1_FBSDE}.  By the assumption  \eqref{hypo_PDE} and  the uniqueness of the solution of the \fbsde \eqref{art1_FBSDE}, the theorem follows. 
  \end{proof}
\begin{Proposition}{ } \label{Acroissement}
 Under the assumptions $ (H1)$ and  $(G4)$, we have the following prior estimates;   there exist  three  positive and continuous functions  $C_1, C_2, C_3 $  such that, for every    $p \geq 2 $ and  $ t,s $ (where $   0\leq s\leq t\leq T$),  
 \begin{eqnarray*}
 \left\{ \begin{aligned} 
  &\text{(i)}  \quad \E ( \Sup_{0\leq s\leq T}\abs{  X^x_s }^p )
 \leq   C_1(T,p)   (1+ |x|^p ) ,   \\  
  &\text{(ii)}  \quad \E ( \Sup_{0\leq t\leq T}\abs{\nabla_x X^x_t }^p )   
 \leq C_2(T,p) , \\
 &\text{(iii)} \quad \E \abs{X^x_t-X^x_s}^{p}
 \leq C_3(T,p)(1+ |x|^{p} )   |t-s|^{p/2}.       
 \end{aligned}
 \right.
 \end{eqnarray*} 
The functions  $ C_{1}, C_{2},C_{3}$  are independent  of  $x,t $ and $s$.  
 \end{Proposition}

\noindent  For more details on the above proposition, we refer you to Ikeda and Watanabe \cite{Ikeda2014stochastic}. 
FBSDEs  and their properties are well documented in the literature. 
Due to their importance, we   need robust approximation  schemes to solve these equations.  
 The Monte Carlo   methods   remain very useful tool to deal with these numerical problems. 
Our work will focus on the numerical solution of the FBSDE \eqref{art1_FBSDE}.  In the sequel, we will for simplicity     work  in  the one dimensional framework. However the result can be extended in high dimensional regimes. 
 We end  up this section  by providing a key Lemma from   Zhang \cite{zhang2004NumSchemeBSDE} which establishes   a  path  regularity result of    the martingale integrand $Z^x$.  This result is known as   the $L_2$-time regularity property  of $Z^x$.     For the reader's convenience,  we recall this result  often used in Section \ref{ConvergenceResults}. 
\begin{Lemma}[Zhang \cite{zhang2004NumSchemeBSDE}]\label{art1_lem_ZhangL2Regularization}
Let $\pi$ be  a partition of the interval $[0,T]$ defined as follows, $   \pi : \, 0=t_0<t_1<\ldots< t_N=T,$  with the  mesh  
$\,\Delta_i:=t_{i+1}-t_i$  and $\abs{\pi} := \Max\{\Delta_i\,; 0\leq i \leq N-1\}$.  
Under the assumption of Theorem 3.1 in    \cite{zhang2004NumSchemeBSDE}, 
 there exists a positive constant $C >0$    independent of $\pi$ such that 
    \begin{equation}  
     \left. 
     \begin{aligned}
     \sum^{N}_{i=1}\E   \int_{t_i-1}^{\ti}
 |  Z^x_s- Z^x_{t_i-1}|^2 +|  Z ^x_s- Z^x_\ti|^2   ds \leq C  (1+|x|^2)|\pi|.  
 \end{aligned} 
     \right.
  \end{equation}
\end{Lemma}
 \newpage
\section{Implicit Backward Euler-Maruyama Scheme}{\label{eulerMaruyama}}
\setcounter{equation}{0} \setcounter{Assumption}{0}
\setcounter{Theorem}{0} \setcounter{Proposition}{0}
\setcounter{Corollary}{0} \setcounter{Lemma}{0}
\setcounter{Definition}{0} \setcounter{Remark}{0}
In this section,  we will review the Euler-Maruyama  scheme   of  the   \fbsde  \eqref{art1_FBSDE}.   As already mentioned in the introduction there are several algorithms to solve BSDEs numerically.  
One of the difficulties is to solve  a  dynamic programming problem  which involves the computation of conditional expectations at each step   across the time interval.
  This computation can be very costly especially in high dimensional problems.
For  most numerical algorithms, it is important to note that  the process $Z^x$ is more challenging to compute than the   process $Y^x$ accurately.
Following the work of   Gobet et al. \cite{gobet2005reg-MC-BSDE}, let us consider the  one-dimensional discrete-time approximation of  the equation~(\ref{art1_FBSDE}). We build  a partition  $\pi$ of  the interval $[0,T]$ defined as:   
$$  \pi : \quad 0=t_0<t_1<\ldots< t_N=T,$$ with the  mesh  
$\,\Delta_i:=t_{i+1}-t_i$  and $\abs{\pi} := \Max\{\Delta_i\,; 0\leq i \leq N-1\}$.   
Let $ (X^{\pi}, Y^{\pi},Z^{\pi})$ be an    approximation of the triplet $(X^x,Y^x,Z^x)$ defined as follows. The forward component $X^x$ of the FBSDE  (\ref{art1_FBSDE})  is approximated  by the  classical Euler-Maruyama    scheme which is  given by 
 \begin{eqnarray*}
 \left\{\begin{aligned} 
 X^{\pi}_{0} &= x   \\
 X^{\pi}_\tiun&=X^{\pi}_{t_{i}} +\Delta_i b(t_i,X^{\pi}_\ti)+\sigma(t_i,X^{\pi}_{t_{i}})(W_\tiun- W_\ti), \quad 0 <i <  N.
 \end{aligned}
      \right.
 \end{eqnarray*}
 By integrating the second equation of  the system (\ref{art1_FBSDE}) from the discretization  time  $t_i$ to $t_{i+1}$,  we obtain 
 $$Y^x_\ti  = Y^x_\tiun + \int_\ti ^{\tiun} f(s,X^x_s,Y^x_s,Z^x_s)ds-\int_\ti ^{\tiun} Z^x_s dW_s.$$
An Euler-Maruyama  approximation of the previous stochastic integral is  defined as  
  $$ Y^{\pi}_\ti  = Y^{\pi}_\tiun + f(\ti,X^{\pi}_\ti,Y^{\pi}_\ti,Z^{\pi}_\ti)\Delta_i - Z^{\pi}_\ti\Delta W_\ti; \qquad
  \Delta W_\ti:=W_\tiun- W_\ti. $$
\noindent By multiplying both sides of the preceding equation with   $\Delta W_\ti$   and   taking the conditional expectations  with respect to $\calF_\ti$ of the preceding equality,  Bouchard and Touzi \cite{bouchard:touzi2004MC-BSDE}  derive   the following   backward  scheme  
\begin{eqnarray*}
 \label{euler}
(S.I)\left\{\begin{aligned}
Y^{\pi}_{N} & =  \phi(X^{\pi}_{T}) ,  \\
Z^{\pi}_\ti &= \frac{1}{\Delta_i}\E [Y^{\pi}_\tiun (W_\tiun-  
 W_\ti) \big|\calF_\ti] , \quad 0 \leq  i \leq N-1, \\
 Y^{\pi}_\ti  & =  \E[Y^{\pi}_\tiun\big|\calF_\ti] 
 +\Delta_i
  f(\ti,X^{\pi}_\ti,Y^{\pi}_\ti,Z^{\pi}_\ti)   , \quad 0 \leq i\leq  N-1.\\ 
  \end{aligned}
      \right.
 \end{eqnarray*} 
The implicit scheme $(S.I)$ is the standard backward Euler-Maruyama  scheme for the backward component of the system \eqref{art1_FBSDE}. Bouchard and Touzi \cite{bouchard:touzi2004MC-BSDE} 
simulate the  conditional expectations  using Malliavin calculus techniques.  
In the spirit of  the Longstaff-Schwartz algorithm for American option pricing,  Gobet et al. \cite{gobet2005reg-MC-BSDE}  have used regression techniques to approach  the solution of the scheme $(S.I)$. Their approach is based on the regression-now technique.  The numerical scheme  $(S.I)$ 
is widely documented  in the literature. The control of the simulation error  has been analyzed in several papers.  
 As with many other existing algorithms,  the implementation of the previous scheme is not explicit. 
 The work of Gobet et al. \cite{gobet2005reg-MC-BSDE} provides an implementation of the numerical scheme $(S.I)$  and derives an analytic convergence rate. 
\begin{Theorem} (Gobet et al. \cite{gobet2005reg-MC-BSDE}).  
\label{Simulation Error}
Under the assumptions of Theorem $1$ in   \cite{gobet2005reg-MC-BSDE},  
there exists a constant $C>0$  such that for  $|\pi|$ small enough, 
\begin{eqnarray*}
\left. \begin{aligned} 
\max_{0\leq i < N }\E |Y^x_{\ti}-Y^{\pi}_\ti |^2 + &  \E  \sum^{N-1}_{i=0} \int_\ti^{\tiun}|Z^x_t- Z^{\pi}_\ti|^2dt \leq C(1+|x|^2)|\pi| \\
    &\QUAD \QUAD+ C\E \abs{\phi(X^x_T)-\phi(X^{\pi}_{t_N})}^2.\qquad \\
\end{aligned}
      \right.
\end{eqnarray*}
\end{Theorem}

\section{Regression-Later Algorithm}
\setcounter{equation}{0} \setcounter{Assumption}{0}
\setcounter{Theorem}{0} \setcounter{Proposition}{0}
\setcounter{Corollary}{0} \setcounter{Lemma}{0}
\setcounter{Definition}{0} \setcounter{Remark}{0}
In order to describe our  regression-later algorithm, we  introduce the pseudo-explicit scheme $(S.II)$ below which governs our  regression-later algorithm.
The regression-later algorithm  is devoted to solve numerically the \fbsde \eqref{art1_FBSDE}. 
 This technique has already been  used   by Glasserman and Yu \cite{glasserman2004simulation} to compute the price of an American option. 
 These authors have shown that the regression-later approach offers advantages  in comparison to  the  regression-now technique. 
  Beutner et  al. \cite{beutner2013fast} have   provide an asymptotic convergence rate of the regression-later technique under some  mild auxiliary   assumption   for   single-period problems. 
 \subsection{Alternative Algorithm}\label{subsec_pseudo_sc
 heme}
 For the sake of clarity, we consider the  one-dimensional discrete time approximation of  the system \eqref{art1_FBSDE}  where, the   partition is given  by $\pi$. 
In the new scheme  $(S.II)$  below,   we denote conventionally  by  $ (X^{\pi}, Y^{\pi},Z^{\pi})$ an approximation of   the triplet $(X^x,Y^x,Z^x)$ via our scheme. 
  It is important to note that the family $ \{(Y^{\pi},Z^{\pi})  \}$ defined below is different from the one defined   in Section  \ref{eulerMaruyama}.
 Only the forward component $X^x$ of the system (\ref{art1_FBSDE})  is approximated  by the   same  Euler-Maruyama  discretization scheme   described in  the previous Section \ref{eulerMaruyama}. The other components are obtained as follows: due to the Markov property of our Euler-Maruyama  scheme, there exist two measurable deterministic functions $u_\ti^\pi$  and $v_\ti^\pi$ such that  
 for every $\ti \in \pi$, one has  $Y^{\pi}_\ti=u_\ti^\pi(X^{\pi}_\ti)$ and  $Z^{\pi}_\ti=v_\ti^\pi(X^{\pi}_\ti)$ almost surely. 
We build   the following  scheme 
\begin{eqnarray*}
\label{eulerGradiant}
(S.II)\left\{\begin{aligned}
Y^{\pi}_{N} &=  \phi(X^{\pi}_{T}), \quad  Z^{\pi}_{N}  = \sigma(T,X^{\pi}_{T})(\nabla_x\phi) (X^{\pi}_{T}) ,\\
 Y^{\pi}_\ti  & =  \E[Y^{\pi}_\tiun\big|\calF_\ti] 
 +\Delta_i
 \E[f(\tiun,X^{\pi}_\tiun,Y^{\pi}_\tiun,Z^{\pi}_\tiun)\big|\calF_\ti] ,    \quad 0 \leq i\leq  N-1, \\
  Z^{\pi}_{\ti} &=\sigma(\ti,X^{\pi}_\ti)^*\nabla_x Y^{\pi}_\ti,    \quad 0 \leq i\leq  N-1. 
 \end{aligned}
      \right.
 \end{eqnarray*}
The couple of discrete processes $ (Y^{\pi}  , Z^{\pi })  $ is obviously adapted to our filtration by definition.    
Regarding the regression-later algorithm, it is also crucial to   control   the error of the numerical estimation of the couple $ (Y^{x}, Z^{x}) $. The error analysis  of  the Euler approximation for the forward  process $X^x$ is well documented and  understood. 
\subsection{\label{algo1}Description of the   Algorithm}
 We notice that   the   solution  of the system \eqref{art1_FBSDE}  has its value in an  infinite dimensional space. 
 In order to compute the conditional expectations in our  algorithm, for each time instance $  i \in \{0, ...,  N \}$,   we define  a    family $ (e^i_j)_{1\leq j \leq k}  $  of truncated orthogonal basis functions of the  space  $  \L^2_1 (\mathcal{F}_\ti) $  where   $(j,k) \in \N^*\times\N^*$. The integer $k$ denotes the number of basis functions. Our  algorithm  admits five major steps of calculations. 
We  define an orthogonal projection onto  the linear subspace   generated by the family  $ (e^i_j)_{1\leq j \leq k}$.
   Each  basis function   is assumed to be at least  differentiable and continuous in the space variable. Orthogonal polynomials are often used in this context. 
 In our numerical implementation,  we will consider a sequence of Hermite polynomials or a sequence of  Laguerre polynomials. We start with  the same partition $\pi$ of  the time interval $[0,T]$ as in the  previous section. We denote by $(Y^{\pi,k},Z^{\pi,k})$ the numerical approximation of the solution    on the discretization grids of the partition  $\pi$. We also assume that we have at our disposal the Euler-Maruyama  approximation of the forward process $X$ on the same discretization grids. 
  {Moreover, the   family of functions $ (e^i_j)_{1\leq j \leq k}$     is selected such that the conditional expectation can be computed exactly}. In other words, during the  regression-later algorithm below,   the conditional expectation term  $\E[e^i(X^{\pi}_\tiun)\big| \mathcal{F}_\ti]$, is assumed to be known  explicitly via the selected basis functions.
 \\
\hrule
 \vspace{4mm}
 \textbf{Description}
 \vspace{2mm}
 \hrule 
 \vspace{2mm}  
\begin{itemize}
\item  \textbf{Initialisation} :  
Approximate the  terminal condition $Y^{\pi,k}_T=Y^{\pi}_T= \phi(X^{\pi}_{T})$.    
\item For $i=(N-1) $ to $0$,   
\begin{itemize}
\item{Compute the vector $\alpha^{i+1}_k\in \R^{k}$ by   projection of $Y^{\pi}_\tiun$}  in \eqref{eulerGradiant} 
    \begin{equation*}
     \left\{ 
     \begin{aligned}
       \text{Find}  
       \quad  & \alpha^{i+1}_k \in \R^{k}
       \quad \text{such that}, \\
       & J (\alpha^{i+1}_k) = \inf_{ \displaystyle \alpha \in \R^{k}} \E \bigg[ \big| \alpha.e^i(X^{\pi}_\tiun)- Y^{\pi,k}_\tiun \big|^2 \bigg ],\\ 
   \end{aligned} 
     \right.
  \end{equation*}
  with   $ e^i= \left( {\begin{array}{*{20}c}
  e^i_1  \\
   . \\
    . \\
  e^i_k   \\
 \end{array} }  \right).$
\item {Compute $Z^{\pi,k}_\tiun$ by  the following formal derivation},
 \begin{equation*}
 Z^{\pi,k}_\tiun= \alpha^{i+1}_k \nabla _x e^i (X^{\pi }_\tiun)  \sigma(t_{i+1},X^{\pi}_\tiun) . 
 \end{equation*}
 \item {Compute the vector $ \beta^{i+1}_k\in \R^{k}$ }  by the following optimization problem,
 \begin{equation*}
  \begin{cases}
       \text{find}  
        \quad  & \beta^{i+1}_k\in \R^{k}
        \quad \text{such that}, \\
     	   & J (\beta^{i+1}_k) = \inf_{ \displaystyle \beta \in \R^{k}} \E \bigg[ \big| \beta.e^i(X^{\pi}_\tiun)- f(t_{i+1},X^{\pi,k}_\tiun,Y^{\pi,k}_\tiun,Z^{\pi,k}_\tiun) \big|^2 \bigg ].\\ 
    \end{cases}
\end{equation*}
\item Evaluate  $$ Y^{\pi,k}_\ti =(\alpha^{i+1}_k+\beta^{i+1}_k \Delta_i). \E[e^i(X^{\pi}_\tiun)\big| \mathcal{F}_\ti]. $$
\end{itemize}
\item \textbf{End of the algorithm}
 \vspace{2mm}
\hrule 
 \vspace{2mm}  
 \end{itemize}

  \vspace{3mm}
\noindent The regression-later scheme presents several advantages. The   primary   advantage is that, at  each time step of the algorithm, the  scheme requires only one conditional expectation computation. The second  advantage is that
  the basis functions $ (e^i_j)_{1\leq j \leq k}$  in the algorithm  are selected such that the conditional expectation can be computed exactly.  Therefore,  the term $\E[e^i(X^{\pi}_\tiun)\big| \mathcal{F}_\ti]$   is known explicitly.    
  These facts could decrease  significantly the time of computation and  accelerate the convergence of the  algorithm especially in  high dimensional  frameworks where  the curse of dimensionality problem occurs. 
As in  Glasserman and Yu \cite{glasserman2004simulation},  the regression-later approach offers many advantages 
and our  numerical implementations   yield  good convergence results in practice.  
\subsection{Convergence}\label{ConvergenceResults}
 By definition,  the couple of discrete processes $ (Y^{\pi}  , Z^{\pi})  $ is   well defined and  adapted to our filtration. Due to the Markov property of the scheme $(S.II)$,  there exist two measurable deterministic functions $u_\ti^\pi$  and $v_\ti^\pi$ such that  
 for every $\ti \in \pi$, one has  $Y^{\pi}_\ti=u_\ti^\pi(X^{\pi}_\ti)$ and  $Z^{\pi}_\ti=v_\ti^\pi(X^{\pi}_\ti)$ almost surely.
Since we are never sure of the accuracy of
a proposed model, it is in general recommended to know how robust the model is.  Regarding the regression-later algorithm, it is  important to  control  the error due to the estimation of the couple $(Y^{\pi},Z^{\pi}) $. 
This control  provides a convergence rate of the regression-later algorithm.  
 By using the scheme $(S.II)$, the following  theorem provides a convergence rate of this error. 
\begin{Theorem}
\label{simulation_errorTHM}
 Under the assumptions  $(H), (G) $ and if  the functions  $ x \mapsto u_\ti^\pi (x) $ are uniformly Lipschitz,   there  exists a positive constant $C$ independent of the partition $\pi$ such that 
\begin{eqnarray}
 \left.  \begin{aligned}
\max_{0\leq i < N }\E |Y^{\pi}_\ti-Y^x_\ti |^2 +  \E  \sum^{N-1}_{i=0} \int_\ti^{t_i+1}|Z^x_s- Z^{\pi}_\ti|^2ds &  \leq    
      {C(1+|x|^2)|\pi|  \quad} \\ 
   \quad  &  +C\E |\phi(X^x_T)-\phi(X^{\pi}_{t_N})|^2. \\
\end{aligned}
      \right.
 \end{eqnarray} 
\end{Theorem}
\noindent In the above theorem, we have assumed the function $u_\ti$ is uniformly Lipschitz for every $\ti \in \pi$.  
We  will argue that this assumption is highly plausible when the mesh $\abs{\pi} := \Max\{\Delta_i\,; 0\leq i \leq N-1\}$ is small enough.  
\begin{proof}[Proof of Theorem ~\ref{simulation_errorTHM}]   The  proof will consist of two parts. In the first part, we  will  prove that:  
\begin{eqnarray*}
 \max_{0\leq i < N }\E |Y^{\pi}_\ti-Y^x_\ti |^2   \leq    
  C(1+|x|^2)|\pi|    +    C\E \abs{\phi(X^x_T)-\phi(X^{\pi}_{t_N})}^2  
 \end{eqnarray*} 
  and   in the second step deduce the existence of the   constant $C>0$ such that
\begin{eqnarray*}
  \E  \sum^{N-1}_{i=0} \int_\ti^{t_i+1}|Z^x_s- Z^{\pi}_\ti|^2ds \leq    
  C(1+|x|^2)|\pi|   +    C\E \abs{\phi(X^x_T)-\phi(X^{\pi}_{t_N})}^2 .
 \end{eqnarray*}  
 During the proof, the constant $C$ may take different values from line to line, but it will be independent from the partition $\pi$.  
Let us  first remark that along the time period $[t_{i}, t_{i+1}] $,
\begin{eqnarray}
\label{discretisation}
\label{equation1}
 \begin{aligned} 
 & Y^x_\ti=Y^x_\tiun + \int_\ti^{{t_{i+1}}} f(X^x_s,Y^x_s,Z^x_s)ds - \int_\ti^{{t_{i+1}}} Z^x_s dW_s.  
 \end{aligned}
 \end{eqnarray}
Taking the conditional expectation with respect to $ \mathcal{F}_\ti$ of the preceding equation 
\begin{eqnarray*}
 \begin{aligned} 
 &Y^x_\ti=\E \left( Y^x_\tiun+ \int_\ti^{{t_{i+1}}} f(X^x_s,Y^x_s,Z^x_s)ds\big|\mathcal{F}_\ti\right).  \\
 \end{aligned}
 \end{eqnarray*} 
\noindent  As defined in the scheme $(S.II)$,  one can  compute an approximation of the  process $Y^x$  at the given time $\ti$  	as the following conditional expectation 
\begin{eqnarray*}
 \begin{aligned} 
  &Y^{\pi}_\ti=\E \left( Y^{\pi}_\tiun+\Delta_i
  f(t_{i+1},X^{\pi}_\tiun,Y^{\pi}_\tiun,Z^{\pi}_\tiun)\big|\mathcal{F}_\ti\right) .
 \end{aligned}
 \end{eqnarray*}
 By a backward induction, one can derive from the preceding equality  that $Y^{\pi}_\ti$ belongs to the space $  \L^2_1(\mathcal{F}_\ti)$.  
 Let us consider  
$$ U_i= \E (|Y^x_\ti-Y_\ti^{\pi} |^2 + \int_\ti^\tiun|Z_s^x-Z^{\pi}_\ti|^2ds ), \quad 0\leq i \leq N -1.$$
We also define $ \delta f^\pi_{i,s}= f(s,X_s,Y^x_s,Z^x_s) - f(t_{i+1},X^{\pi}_\tiun,Y^{\pi}_\tiun,Z^{\pi}_\tiun), \,\, s \in [\ti,\tiun].$
\begin{Remark}{\label{art1_remark_MRT}} 
 \noindent \begin{itemize}
\item  $(Y^x_{\ti} -Y^{\pi}_\ti ) $ and $\displaystyle{(\int_{\ti}^{\tiun}(Z^x_s - Z^{\pi}_\ti) dW_s)}$  are uncorrelated.   
 \item  By the martingale representation theorem, there exists an  $(\mathcal{F}_s )_{t_i \leq s \leq \tiun}$- adapted and square  integrable process 
$\displaystyle {(\bar Z^{\pi}_t)_{t_i \leq t \leq \tiun}}$ and $\displaystyle {(Y^{\pi}_t)_{t_i \leq t \leq \tiun}}$ such that for $t \in [ t_i,  t_{i+1}]$,  
\begin{eqnarray} \label{eq_MartingaleRepresentoF_Z}
\begin{aligned}
 Y^{\pi}_{t} = Y^{\pi}_\tiun + \int_{t}^{{t_{i+1}}} f(t_{i+1},X^{\pi}_\tiun,Y^{\pi}_\tiun, Z^{\pi}_\tiun)ds - \int_{t}^{{t_{i+1}}}\bar Z^{\pi}_sdW_s.
 \end{aligned}
 \end{eqnarray}
 \item  The   process  $\displaystyle \bar Z^{\pi}$ is   c\`adl\`ag   and is equal to     $\displaystyle Z^{\pi}$ only on the time instances of the partition $\pi$. 
\end{itemize}
 \end{Remark}

From the equations \eqref{discretisation} and \eqref{eq_MartingaleRepresentoF_Z}, 
\begin{eqnarray} \label{inegalite_explik}
 \begin{aligned} 
  Y_{\ti}^x -Y^{\pi}_\ti  +  \int_{\ti}^{\tiun}(Z^x_s -  Z^{\pi}_\ti)& dW_s   = Y^x_{\tiun}  -Y^{\pi}_\tiun +\int_\ti^\tiun \delta f^\pi_{i,s}ds \\ 
  & \QUAD +\int_\ti^\tiun( \bar Z^{\pi}_s- Z^{\pi}_\ti) dW_s.
 \end{aligned}
 \end{eqnarray}
 %
 From the inequality \ref{art1_Lem_YoungInequality}, we have   for all $a,b,c \in \R$ and $\alpha >0$
  \begin{eqnarray}\label{ineq_YoungInequality_explicit}
 (a+b+c)^2 \leq  (1+\alpha) a^2 + (1+ \frac{2}{\alpha}) b^2 + (1+\alpha) c^2 + 2 ac.   
 \end{eqnarray} 
Using the equation \eqref{inegalite_explik},
\begin{eqnarray*}
 \begin{aligned}
U_i =& \E \Big [Y^x_{\tiun}  -Y^{\pi}_\tiun +\int_\ti^\tiun \delta f^\pi_{i,s}ds +\int_\ti^\tiun( \bar Z^{\pi}_s- Z^{\pi}_\ti) dW_s  \Big ]^2.
 \end{aligned}
  \end{eqnarray*}
By the  It\^o isometry formula and the quadratic inequality \eqref{ineq_YoungInequality_explicit}, we derive from the above remark  that    for every $\epsilon >0$,  
 \begin{eqnarray*}
 \begin{aligned}
U_i &\leq \E \Big\{(1+ \Delta_i/\epsilon)| Y_\tiun^{\pi}-Y^x_\tiun |^2  +(1+ \Delta_i/\epsilon)  \int_\ti^\tiun|\bar Z_s^{\pi}-Z^{\pi}_\ti|^2ds \\
  &    \quad  +(1+  {2}\epsilon/\Delta_i)(\int_\ti^\tiun f(s,X^x_s,Y_s,Z_s) - f(t_{i+1},X^{\pi}_\tiun,Y^{\pi}_\tiun,Z^{\pi}_\tiun)ds)^2 \Big\} \qquad \\
    &    \qquad  \qquad     +2   \E \Big \{ ( Y_\tiun^{\pi}-Y^x_\tiun)  \, (  \int_\ti^\tiun (\bar Z_s^{\pi}-Z^{\pi}_\ti)dW_s)\Big\}. 
 \end{aligned}
  \end{eqnarray*}
We know that $\displaystyle M_t= \int_\ti^t (\bar Z_s^{\pi}-Z^{\pi}_\ti)dW_s     
,  \, t \in [\ti, \tiun]$ defines  a   martingale in the Brownian filtration. 
Plugging the  equation \eqref{inegalite_explik} into the last term of the previous inequality and taking the conditional expectation according to $\calF_{t_i}$ and  noticing that $(\int_t^\tiun \delta f^\pi_{i,s}ds )_{t_i \leq t \leq \tiun}$ is  of finite variation, we have from the  It\^o isometry formula, 
 \begin{eqnarray}
 \label{inegalite1.1}
 \begin{aligned}
U_i &\leq \E \Big\{(1+ \Delta_i/\epsilon) | Y_\tiun^{\pi}-Y^x_\tiun |^2  +  (1+ \Delta_i/\epsilon)  \int_\ti^\tiun|\bar Z_s^{\pi}-Z^{\pi}_\ti|^2ds \\
  &   \quad +(1+2\epsilon/\Delta_i)(\int_\ti^\tiun f(s,X^x_s,Y_s^x,Z^x_s) - f(t_{i+1},X^{\pi}_\tiun,Y^{\pi}_\tiun,Z^{\pi}_\tiun)ds)^2\\
   &    \QUAD \qquad \qquad - 2  \E    \int_\ti^\tiun  (Z^x_s - \bar Z^{\pi}_s)  (\bar Z_s^{\pi}-Z^{\pi}_\ti) ds\Big\}.  
 \end{aligned}
  \end{eqnarray}  
  By noticing that:  
  $(  Z^x_s - \bar Z^{\pi}_s)  (\bar Z_s^{\pi}-Z^{\pi}_\ti) = (  Z^x_s -Z^{\pi}_\ti )  (\bar Z_s^{\pi}-Z^{\pi}_\ti)    -   (\bar Z_s^{\pi}-Z^{\pi}_\ti)^2  $
 and  from  the inequality  $\,    2ab  \leq \frac{1}{\theta } a ^2 +\theta  b^2$ (with  $a,b \in \R, \,\,   \text{for any}\,\,  \theta > 0 $), we have by  the relation      \eqref{inegalite1.1} 
 \begin{eqnarray*}
 \begin{aligned}
U_i &\leq \E \Big\{(1+ \Delta_i/\epsilon) | Y_\tiun^{\pi}-Y^x_\tiun |^2  +  ( 3 + \Delta_i/\epsilon+ \theta) \int_\ti^\tiun|\bar Z_s^{\pi}-Z^{\pi}_\ti|^2ds \\
  &   \quad +   (1+2\epsilon/\Delta_i) (\int_\ti^\tiun f(s,X^x_s,Y^x_s,Z^x_s) - f(t_{i+1},X^{\pi}_\tiun,Y^{\pi}_\tiun,Z^{\pi}_\tiun)ds)^2  \\
   &    \QUAD \qquad \qquad +   \frac{1}{\theta}   \E   \int_\ti^\tiun  |  Z^x_s -Z^{\pi}_\ti| ^2 ds\Big\}.    
 \end{aligned}
  \end{eqnarray*}  
By  the H\"{o}lder inequality, we have 
  \begin{eqnarray*}
  \begin{aligned}
U_i &\leq \E \Big\{(1+ \Delta_i/\epsilon)| Y_\tiun^{\pi}-Y^x_\tiun |^2  +( 3 + \Delta_i/\epsilon+ \theta)  \int_\ti^\tiun|\bar  Z_s^{\pi}-Z^{\pi}_\ti|^2ds \\
  &    \qquad  + ( \Delta_i+2\epsilon )\int_\ti^\tiun | f(s,X^x_s,Y^x_s,Z^x_s) - f(t_{i+1},X^{\pi}_\tiun,Y^{\pi}_\tiun,Z^{\pi}_\tiun)|^2 ds  \\
     &    \QUAD \qquad \qquad +   \frac{1}{\theta} \E   \int_\ti^\tiun  |  Z^x_s -Z^{\pi}_\ti| ^2 ds \Big\}.   
 \end{aligned}
  \end{eqnarray*} 
 By  the Lipschitz  condition of the driver function  $f$ and the     inequality \eqref{ineq_YoungInequality}, 
  \begin{eqnarray}
 \label{inegalite3}
  \begin{aligned}
U_i &\leq \E \Bigg\{(1+ \Delta_i/\epsilon)| Y_\tiun^{\pi}-Y^x_\tiun |^2  +  ( 3 + \Delta_i/\epsilon+ \theta) \int_\ti^\tiun|\bar  Z_s^{\pi}-Z^{\pi}_\ti|^2ds \\
  &   \quad +  2K^2(\Delta_i+2\epsilon ) \bigg ( \int_\ti^\tiun  |Y^x_s-Y^{\pi}_\tiun|^2ds + \int_\ti^\tiun  |Z^x_s-Z^{\pi}_\tiun|^2ds\bigg ) \\ 
    &   \quad \quad \quad  +   \frac{1}{\theta}     \E   \int_\ti^\tiun  |  Z^x_s -Z^{\pi}_\ti| ^2 ds    + 2K^2(\Delta_i+2\epsilon )\int_\ti^\tiun  |X^x_s-X^{\pi}_\tiun|^2ds  \Bigg\}.  
 \end{aligned} 
  \end{eqnarray}  
  It is known from   for instance Lemma $3.2$  in Zhang  \cite{zhang2004NumSchemeBSDE} or Proposition $5$ in Gobet et al. \cite{gobet2007error}  
  and the result $(iii)$ of Proposition \ref{Acroissement}  that,  there exists a constant  $C>0$  such that 
 \begin{eqnarray} \label{Accroisment_X}
  \begin{aligned}
     \E \int_\ti^\tiun|X^x_s - X^{\pi}_\tiun|^2 ds &\leq 2 \int_\ti^\tiun \E |X^x_s - X^x_\tiun|^2 ds + 2\Delta_i \E |X^x_\tiun - X^{\pi}_\tiun|^2 \\
      & \leq   C(1+|x|^2)|\pi|^2.
 \end{aligned} 
  \end{eqnarray} 
Moreover, we have  
\begin{itemize}  
  \item   $\quad {|Z^x_s-Z^{\pi}_\tiun|= |Z^x_s-Z^x_{s+\Delta_i} + Z^x_{s+\Delta_i}-Z^{\pi}_\tiun|}.$  
 \item    $  \displaystyle{\quad\int_\ti^\tiun|Y^x_s - Y^{\pi}_\tiun|^2 ds \leq 2\int_\ti^\tiun|Y^x_s - Y^x_\tiun|^2 ds + 2\Delta_i |Y^x_\tiun - Y^{\pi}_\tiun|^2}$.
 \end{itemize}
  By the preceding decomposition, we obtain  from  the inequality  \eqref{inegalite3}, 
  \begin{eqnarray}
 \label{inegalite4}
  \begin{aligned}
U_i &\leq \E \Bigg\{C^K_{\epsilon,i}\,\, | Y_\tiun^{\pi}-Y^x_\tiun |^2  + 4K^2(\Delta_i+2 \epsilon )\int_\ti^\tiun |Z^x_{s+\Delta_i}-Z^{\pi}_\tiun|^2ds  \\
  &   \qquad + 4K^2(\Delta_i+2\epsilon )\int_\ti^\tiun  |Y^x_s-Y^x_\tiun|^2ds+  ( 3 + \Delta_i/\epsilon+ \theta)  \int_\ti^\tiun|\bar  Z_s^{\pi}-Z^{\pi}_\ti|^2ds \qquad \\ 
     & \quad \quad   +   4K^2(\Delta_i+2\epsilon )\int_\ti^\tiun |Z^x_s-Z^x_{s+\Delta_i}|^2ds    + \frac{1}{\theta} \E   \int_\ti^\tiun  |  Z^x_s -Z^{\pi}_\ti| ^2 ds \Bigg\} \\
    & \QUAD \QUAD \QUAD     +   C_\epsilon(1+|x|^2)|\pi|^2.   
 \end{aligned}
  \end{eqnarray}    
where $ C^K_{\epsilon,i}=   (1+ \Delta_i/\epsilon+4 \Delta_iK^2(\Delta_i+2\epsilon))$ and $C_\epsilon=   4CK^2(\Delta_i+2\epsilon ) $. 
Clearly    
 \begin{eqnarray}
   \label{technic}
  \begin{aligned}
 &\E\int_\ti^\tiun |Z^x_{s+\Delta_i}-Z^{\pi}_\tiun|^2ds    = \E\int_\tiun^{t_{i+2}} |Z^x_{s }-Z^{\pi}_\tiun|^2ds. \\
  \end{aligned}
  \end{eqnarray} 
From the equality \eqref{technic}, the inequality \eqref{inegalite4} becomes 
 \begin{eqnarray}
 \label{inegalite5}
  \begin{aligned}
U_i &\leq \E \Bigg\{C^K_{\epsilon,i} \,\,| Y_\tiun^{\pi}-Y^x_\tiun |^2  + 4K^2(\Delta_i+2\epsilon ) \int_\tiun^{t_{i+2}} |Z^x_{s }-Z^{\pi}_\tiun|^2ds \\
  &   \qquad \qquad   +   C_\epsilon(1+|x|^2)|\pi|^2     +  4K^2(\Delta_i+2\epsilon )\int_\ti^\tiun  |Y^x_s-Y^x_\tiun|^2ds \quad\quad   \\
     &  \qquad    + 4K^2(\Delta_i+\epsilon )\int_\ti^\tiun |Z^x_s-Z^x_{s+\Delta_i}|^2ds   +    ( 3 + \Delta_i/\epsilon+ \theta)    \int_\ti^\tiun|\bar  Z_s^{\pi}-Z^{\pi}_\ti|^2ds   \\ 
        & \QUAD \QUAD        + \frac{1}{\theta}  \int_\ti^\tiun  |  Z^x_s -Z^{\pi}_\ti| ^2 ds \Bigg\}. 
 \end{aligned}
  \end{eqnarray}  
From Lemma 3.2 in  \cite{zhang2004NumSchemeBSDE},  there exists a positive constant $C$ such that 
    \begin{eqnarray} \int_\ti^\tiun  |Y^x_s-Y^x_\tiun|^2ds \leq   C(1+|x|^2)|\pi|^2\label{Accroisment_Y}. \end{eqnarray}  
Inserting the inequality \eqref{Accroisment_Y} into     \eqref{inegalite5} and setting    $   (\epsilon ;\frac{1}{\theta}) = ( \frac{1}{16K^2}; \frac{1}{2}),$  we   derive a constant  $C>0$ such that 
\begin{eqnarray*}
  \begin{aligned}
\tilde{U}_i &\leq  (1+ C\Delta_i  ) \tilde{U}_{i+1} +  C( 1 + \Delta_i) \bigg (   (1+|x|^2)|\pi|^2   + \E\int_\ti^\tiun |Z^x_s-Z^x_{s+\Delta_i}|^2ds\bigg ) \quad     \\
      & \QUAD  \QUAD     \quad +C(1+   \Delta_i) \E \int_\ti^\tiun   |\bar  Z_s^{\pi}-Z^{\pi}_\ti|^2ds,  
 \end{aligned}
 \end{eqnarray*}   
 where   $  \displaystyle \tilde{U}_i = U_i - \frac{1}{2} \E   \int_\ti^\tiun  |  Z^x_s -Z^{\pi}_\ti| ^2 ds  $.  
From Lemma \ref{art1_lem_ZhangL2Regularization}  and   Lemma \ref{art1_Gronwall1},  there exists a constant $	C>0$ such that  for $|\pi|$ small enough,   
 \begin{eqnarray}
 \left.\label{inegality7}
  \begin{aligned}
 \max_{0\leq i \leq N} \tilde{U}_i \leq C\E(\phi(X^x_{T})-\phi(X^{\pi}_{T})^2  
+ C\sum^{N-1}_{i=0} & \E \int_\ti^\tiun\abs{\bar Z_s^{\pi}-Z^{\pi}_\ti}^2ds \\
 &  +C(1+|x|^2)|\pi|.
  \end{aligned}
\right. 
 \end{eqnarray}
The following argument concludes our proof.  Interval-by-interval, 
 given that   $ \bar Z^\pi_\ti = Z^\pi_\ti$ and the result of  Lemma \ref{art1_lem_ZhangL2Regularization}, 
there exists a positive constant $C>0$ independent of $\pi$ such that 
  \begin{equation}
  \label{eq_zhang}\sum^{N-1}_{i=0}\E   \int_\ti^{\tiun}
 |\bar Z^\pi_s- Z^\pi_\ti|^2  ds \leq C(1+|x|^2)|\pi|.   \end{equation}
Inserting the inequality \eqref{eq_zhang}    into  the  inequality \eqref{inegality7}, we obtain
  \begin{equation}
  \label{inegality8.0}
  \max_{0\leq i \leq N}   \tilde U_i \leq C\E(\phi(X^x_T)-\phi ( X^{\pi}_{T}))^2 + C(1+|x|^2)|\pi|.
  \end{equation}
  In particular, one can derive the following  inequality which     completes the first step of the proof of the theorem  
  \begin{equation}
  \label{inegality8}
  \max_{0\leq i \leq N} \E|Y_\ti^{\pi}-Y^x_\ti|^2 \leq C \E(\phi(X^x_T)-\phi(X^{\pi}_{T}))^2 + C(1+|x|^2)|\pi|. 
  \end{equation}
From  the  inequality \eqref{Accroisment_Y}  and  Lemma \ref{art1_lem_ZhangL2Regularization},    the  inequality \eqref{inegalite5} becomes for  $ |\pi|$ small enough   and  choosing    $(\epsilon,\frac{1}{\theta}) = ( \frac{1}{32K^2}, \frac{1}{2})$,
\begin{align*}
 \tilde U_{i-1}+ \frac{1}{4} \E\int_\ti^{\tiun}|Z^x_s- Z^{\pi}_\ti|^2 ds &\leq (1+C\Delta_i )  \tilde U_{i}   +C(1+|x|^2)|\pi|^2 \\
  & \qquad  +  C(1+   \Delta_i  ) \E\int_{t_{i-1}}^{\ti}|\bar  Z_s^{\pi}-Z^{\pi}_{t_{i-1}}|^2ds,  
\end{align*}
 where    $C>0$ and we used  that
$$\E | Y_\ti^{\pi}-Y^x_\ti |^2 =  \tilde U_i -  \frac{1}{2} \E   \int_\ti^\tiun  |  Z^x_s -Z^{\pi}_\ti| ^2 ds.$$ 
Summing both sides of  the previous inequality     over  the variable $i$ from $1$ to $N-1$, and using the inequality \eqref{eq_zhang},  there exists a positive constant $C>0$ independent of 
 $\pi$ such that
 \begin{equation*}
 \label{inegality9}
 \sum^{N-1}_{i=1}  \tilde U_{i-1}+ \frac{1}{4} \E  \sum^{N-1}_{i=1} \int_\ti^{\tiun}|Z^x_s- Z^{\pi}_\ti|^2ds\leq 
 \sum^{N-1}_{i=1}(1+C\Delta_i ) \tilde U_{i} + C(1+|x|^2)|\pi|.
  \end{equation*}
 We deduce from  the previous relation  and the inequality \eqref{inegality8.0}  that there exists a constant $C>0$ independent of  $\pi$ such that
  \begin{eqnarray}
   \label{inegality10}
 \begin{aligned}
 &   \sum^{N-1}_{i=0}  \E\int_\ti^{\tiun}|Z^x_s- Z^{\pi}_\ti|^2ds \leq    
      C(1+|x|^2)|\pi| +C\E \abs{\phi(X^x_T)-\phi(X^{\pi}_{t_N})}^2.
\end{aligned}
\end{eqnarray}
The last  relation  \eqref{inegality10}  and the inequality   \eqref{inegality8}  conclude.
\end{proof}
\noindent  \textbf{Discussion}: Lipschitz Continuity. \\ 
 In Theorem   \ref{simulation_errorTHM}, we have assumed that the function $u_\ti$ is uniformly Lipschitz for any $\ti \in \pi$.   In the following, we will argue   that such condition is  highly plausible.  
We consider the same partition $\pi$ of  the interval $[0,T]$ as described in the algorithm $ (S.II)$. We recall that $Y^{\pi}_\ti$ defines the Euler approximation  of  $Y^{x}_\ti$ (the exact process at the time step $\ti$). As introduced previously,  
\begin{equation*}
 Y^{\pi}_\ti=\E \left( Y^{\pi}_\tiun+\Delta_i
  f(t_{i+1},X^{\pi}_\tiun,Y^{\pi}_\tiun,Z^{\pi}_\tiun)\big| 
  \mathcal{F}_\ti\right) .
 \end{equation*}
By the martingale representation theorem, there exists an $(\mathcal{F}_s )_{t_i \leq s \leq \tiun}$ adapted and square integrable process 
$(\bar Z^{\pi}_s)_{t_i \leq s \leq \tiun}$ such that 
\begin{equation}
\label{EulerContinu}
Y^{\pi}_{t} =Y^{\pi}_\tiun + \int_t^{ \tiun} f(\tiun,X^{\pi}_\tiun,Y^{\pi}_\tiun,Z^{\pi}_\tiun)ds - \int_t^{ \tiun}  \bar Z^{\pi}_s  dW_s, \, t_i \leq t \leq  t_{i+1}.
\end{equation}
The preceding  representation   can be seen as a continuous version of a  BSDE  on the  time interval    $  [\ti, \tiun] $.  Let us  introduce the continuous Euler discretization of the process  of $X^x$ in  the system  \eqref{art1_FBSDE} given by
 \begin{eqnarray}\label{EulerContinuForward}
 \begin{aligned}
 X^{\pi}_s &= X^{\pi}_\ti +  \int_\ti^{s}   b( \ti , X^{\pi}_\ti) du +  \int_\ti^{s}\sigma ( \ti ,X^{\pi}_\ti) dW_u,     \quad  s \in [\ti, \tiun].  
 \end{aligned}
 \end{eqnarray}
 Let us consider   $ X^{\pi,i}, (i=1,2) $  two  solutions of \eqref{EulerContinuForward}  associated with  two  initial conditions $x_i, (i=1,2)$.   We also associate with  $X^{\pi,i}$, its corresponding solutions  $(Y^{\pi,x_i}, \bar Z^{\pi,x_i}),i=1,2 $  of the equation \eqref{EulerContinu}. 
Let us  define the following   terms 
 $$ \Delta Y^{1,2} _t   :=  Y^{\pi,x_1}_t  -  Y^{\pi,x_2}_t  \quad \text{and} \quad \Delta X^{1,2} _t   :=  X^{\pi,1}_t  -  X^{\pi,2}_t .$$    
 As highlighted above,  due to the Markov property of our Euler scheme, there exist   two measurable deterministic functions $u_\ti^\pi$ and $v_\ti^\pi$ such that  
 for every $\ti \in \pi$ one has,  $Y^{\pi}_\ti=u_\ti^\pi(X^{\pi}_\ti)$ and $  Z^{\pi}_{\ti}  =v_\ti^\pi (X^{\pi}_\ti) $  almost surely.  
 For $i=N$, the function   $ x \mapsto u_{T}(x)=\phi(x)  $ is Lipschitz by assumption.  
 We  now suppose that the function   $u_\tiun^\pi$ is Lipschitz in the space variable  with    $ C_{i+1}$ its Lipschitz constant.  We will  show that   $u_\ti^\pi$ is Lipschitz. 
  Applying It\^o's  formula   to the term  $|Y^{\pi,x_1}   -  Y^{\pi,x_2}|^2$ and taking the expectation, we obtain 
%
%
 \begin{eqnarray*} 
 \begin{aligned}
  \E |Y^{\pi,x_1}_t  -  Y^{\pi,x_2}_t |^2   +  \E\int_{t}^\tiun  |\bar  Z^{\pi,x_1}_s-\bar  Z^{\pi,x_2}_s|^2 ds & = \E |Y^{\pi,x_1}_\tiun  -  Y^{\pi,x_2}_\tiun |^2  \\  
  & \,\, +  2 \E \int_{t}^\tiun (Y^{\pi,x_1}_s -  Y^{\pi,x_2}_s) \delta f^{\pi}_i ds.  
 \end{aligned}
 \end{eqnarray*}
where 
 $\delta  f^{\pi}_i =f(\tiun, X^{\pi,x_1}_\tiun,Y^{\pi,x_1}_\tiun,Z^{\pi,x_1}_\tiun)- f(\tiun,X^{\pi,x_2}_\tiun,Y^{\pi,x_2}_\tiun,Z^{\pi,x_2}_\tiun) $.  From the assumption $(H2)$ and the inequality  $ ab \leq  \frac{1}{2\alpha} a^2+\frac{1}{2}\alpha b^2$, $\alpha> 0$ 
 \begin{eqnarray}\label{eq_M3}
\begin{aligned}
  \E |\Delta Y^{1,2} _t |^2   +  \E\int_{t}^\tiun  |\bar  Z^{\pi,x_1}_s-&\bar  Z^{\pi,x_2}_s|^2 ds     \leq    (1+ K  \Delta_i \alpha  )  \E |\Delta Y^{1,2} _\tiun|^2 \\ 
 &  +\frac{2K}{\alpha}  \int_{t}^\tiun  \E|\Delta Y^{1,2} _s|^2 ds +  \alpha K \Delta_i \E |Z^{\pi,x_1}_\tiun  -  Z^{\pi,x_2}_\tiun |^2.
 \end{aligned}
 \end{eqnarray}
 We point out that on the interval  $[0,T]$,   the process  $(\bar Z^\pi_s)_{0 \leq s\leq T}$  defines a  c\`adl\`ag process.  Given  the fact that    $u_\tiun^\pi$ is Lipschitz and 
 $ \bar Z^\pi_\tiun = Z^\pi_\tiun$,   by  Lemma \ref{art1_lem_ZhangL2Regularization} and the quadratic  inequality \eqref{ineq_YoungInequality},   there exist   two  finite and positive constants $c_i^1$ and $c_i^2 > 0$ such that  
\begin{eqnarray*} 
\begin{aligned}
\alpha K \E\int_\ti^{\tiun}|Z^{\pi,x_1}_\tiun  - Z^{\pi,x_2}_\tiun |^2ds \leq &  \alpha K c_i^1(1+|x_1|^2) |\pi|^2 \\ 
 & \qquad        + 3 \alpha K \int_\ti^{\tiun}\E |\bar  Z^{\pi,x_1}_s  - \bar  Z^{\pi,x_2}_s |^2 ds\\
 & \QUAD     + \alpha K c_i^2(1+|x_2|^2) |\pi|^2 . 
 \end{aligned}
 \end{eqnarray*}
Neglecting the terms with $|\pi|^2$, and inserting ( for  $ \alpha =\frac{1}{6K}$ )  the last inequality  into \eqref{eq_M3}, we have    
 \begin{eqnarray*} 
\begin{aligned}
  \E |\Delta Y^{1,2} _t |^2   +  \frac{1}{2}   \E\int_{t}^\tiun  |\bar  Z^{\pi,x_1}_s-\bar  Z^{\pi,x_2}_s|^2 ds  &   \leq    (1+ \frac{1}{6} \Delta_i  )  \E |\Delta Y^{1,2} _\tiun|^2 \\   
& \qquad +12K^2\int_{t}^\tiun  \E|\Delta Y^{1,2} _s|^2 ds.
 \end{aligned}
 \end{eqnarray*}
 In particular,
 \begin{eqnarray}\label{eq_M5}
\begin{aligned}
  \E |\Delta Y^{1,2} _t |^2    &   \leq    (1+ \frac{1}{6} \Delta_i  )  \E |\Delta Y^{1,2} _\tiun|^2   
+12K^2 \int_{t}^\tiun  \E|\Delta Y^{1,2} _s|^2 ds.
 \end{aligned}
 \end{eqnarray}
  During our backward induction proof, we have assumed above that the function $u_\tiun^\pi$ is Lipschitz. From the equation  \eqref{EulerContinuForward}, we have the following classical estimates
 $$ \E |X^{\pi,x_1}_\tiun  -  X^{\pi,x_2}_\tiun |^2 
 \leq   (1+ C\Delta_i) |x_1-x_2|^2.$$ 
    Gronwall's inequality  from Lemma   \ref{art1_Gronwall2}  applied to   the function $ t \mapsto   \E |\Delta Y^{1,2} _t |^2$ with $t \in [\ti,\tiun) $,  we have from \eqref{eq_M5}    
 \begin{eqnarray*}  
\begin{aligned}
 \E |\Delta Y^{1,2} _t |^2\leq &  ( 1+  \frac{1}{6} \Delta_i ) (1+ C\Delta_i )     C^2_{i+1}\exp( 12K^2\Delta_i) |x_1-x_2|^2 .      
 \end{aligned}
 \end{eqnarray*}  
 We recall that our objective is to prove that the function $u_\ti^\pi$ is Lipschitz with a uniform Lipschitz constant in the space variable.
 We have   
  \begin{eqnarray*}  
\begin{aligned}
  |u_\ti^\pi(x_1)-u_\ti^\pi(x_2) |^2\leq & C_i^2|x_1-x_2|^2,     
 \end{aligned}
 \end{eqnarray*}     
where $C_i^2 =  ( 1+  \frac{1}{6} \Delta_i ) (1+ C\Delta_i ) C^2_{i+1}\exp( 12K^2\Delta_i)$. It is then enough to show that, the positive constant $C_i$ is uniformly bounded to conclude the backward induction result. Let us first remark that in the neighborhood of zero, there exists a positive constant $C$ such that $ \exp(\Delta_i) \leq (1+C \Delta_i)$. 
 Hence, for $\Delta_i$ small enough there exists  a positive constant $C$ such that  
 $$C_i^2 \leq    (1+  C \Delta_i )C^2_{i+1}+ C \Delta_i . $$
 By Lemma \ref{art1_Gronwall1}, we have the following uniformly bounded inequality 
$$  \quad   \max_{0\leq i \leq N } C_i^2  \leq   e^{ C T} ( C^2_\phi + CT ), $$
  where $C_\phi$ is the Lipschitz constant of the function $\phi$ in the \fbsde \eqref{art1_FBSDE}. 
Finally,  
 \begin{eqnarray*} 
\begin{aligned}
  |u_\ti^\pi(x_1)-u_\ti^\pi(x_2) |^2\leq &  e^{ C T} ( C^2_\phi + CT )|x_1-x_2|^2.  
 \end{aligned}
 \end{eqnarray*}       
 This completes the induction. 
 %
From the previous inequality,  the function  $ x \mapsto u_\ti^\pi (x) $ is Lipschitz with a uniform Lipschitz constant $  e^{ C T} ( C^2_\phi + CT )$.

\begin{Remark}{\label{art1_remark_MallivinLipschitz}}
 A similar result of the  Lipschitz continuity   can  be obtained with  the semi-group of  $X^\pi$ through  the integration by parts  formula of  Malliavin Calculus (Definition 1.3.1 in Nualart \cite{nualart2006malliavin}). 
\end{Remark}
\newpage
 

\section{Applications}
\setcounter{equation}{0} \setcounter{Assumption}{0}
\setcounter{Theorem}{0} \setcounter{Proposition}{0}
\setcounter{Corollary}{0} \setcounter{Lemma}{0}
\setcounter{Definition}{0} \setcounter{Remark}{0}
In this section, we   provide two numerical experiments to illustrate the performance of the regression-later  algorithm; the first in the context of option pricing  and  the second in the case  where  the terminal condition is a functional of Brownian motion.   
 The first example  is generally connected to the numerical approximation of   a linear or a nonlinear BSDE.    
 
\vspace{0.5cm} 
\noindent
  BSDEs appear  in numerous problems in finance, in insurance and especially in stochastic control.  
 A frequent problem in finance or in insurance is  the problem of the valuation of a contract and the  risk management of a portfolio which becomes  increasingly complex. Linear and nonlinear BSDEs  appear naturally in these situations. The interested reader can consult  the paper of El Karoui  et  al.
  \cite{ElKaroui1997BSDEinFinance}, Delong  \cite{delong2013backward}, 
  Cheridito et al.    \cite{cheridito2007second},  Duffie et al.  \cite{duffie1992stochastic}, Hamad\`ene et al. \cite{hamadene2007starting} and the references therein for further details.   Many problems in finance or in insurance are nonlinear. We will discuss in the first example the linear case and show how fast our algorithm converges. 
In financial markets the most popular contracts of derivative securities are European and American Call and Put  options. 

\vspace{0.5cm} 
\noindent
 In our  first example,  we will   evaluate standard European options.  The algorithm can   also be  applied to compute the price of some non-path dependent insurance contracts. 
In our implementation, we will consider the orthogonal  Laguerre polynomial  family  as basis in order to solve the conditional expectations problems in our algorithm. 
%
  %
 %
\subsection*{Application 1: Pricing}
Our market model is composed of  two financial assets:  $S$ (risky asset) and $S^0$ (risk-less asset). Let 
 \begin{equation*}
  \left\{ \begin{aligned}
  &   S_t \,\text{ is the price of $S$ a the time   $t$ } \\
  &    S_t^0  \,\text{ is the price of $S^0$  a the time  $t$}. 
       \end{aligned}
      \right.
\end{equation*}
\noindent Based on their assumptions, Black and Scholes have modelled the dynamic of the risky  asset $S$  as 
  a geometric Brownian motion.
  We denote by   the constant $r$ the daily interest rate which is  assumed to be  constant. The process  $S^0$ is governed by the following differential equation: $dS_t^0= rS_t^0dt$ with the initial condition  
  $S_0^0=1$. We have explicitly  $S_t^0= e^{rt}$.  
 The process   $S_t$ follows the following linear SDE  with constant coefficients,  
 \begin{equation}
 \left\{ \begin{aligned}
	     &\frac{dS_t}{S_t}  = \mu dt+\sigma dW_t,   \\
       &S_0 = x,\quad 		        
       \end{aligned} 
        \label{ED1} 
      \right.
\end{equation}
where 
 $\mu\in \R $   is a constant drift coefficient which represents  the expected rate of return of $S$,  
  $S_0 $ is the initial value of the risky asset $S$ and   $\sigma$  is a constant  positive volatility coefficient. By It\^o's Lemma, one can show   that the explicit solution of (\ref{ED1})  is given by    
$ S_t= x e^{(\mu - \frac{1}{2}\sigma^2 t)+ \sigma W_t}.$
Let us consider a European Call option on the risky asset $S$ with  characteristics $(K,T)$,  where  $T$ is the maturity date  and $K$  is the strike value of the contract. The  seller of the Call option is committed to pay to the holder  the sum    $(S_T - K)^+$  which represents the profit that allows to  exercise   the option. We build the following portfolio: 
   at the  time instance $t$, we invest a $\Delta_t$ part of the risky asset and  a $\beta_t$ part of the non-risky asset. Denoting $Y$  the wealth process, we have  at  time   $t$ 
		   \begin{equation*}
		         Y_t= \Delta_tS_t+\beta_t S_t^0.
		   \end{equation*}
A main assumption is that our strategy is self-financing and in a context of  continuously trading for the agent,  a  mathematical translation   is  given by  
 \begin{equation*}
	 dY_t= \Delta_tdS_t+\beta_t dS_t^0.
 \end{equation*}
 Denoting $\theta = \frac{\mu-r}{\sigma}$ and  $Z_t= \sigma \Delta_t S_t$,  
  the triplet $(S_t, Y_t,Z_t)$ solves the following system  
\begin{eqnarray*}
 (E_1)\left\{ \begin{aligned}
-dY_t  &= f(t,S_t,Y_t,Z_t)dt - Z_t dW_t, \qquad Y_T  = \phi(S_T),  \\
    dS_t &=\mu S _t dt  + \sigma S _t d W_t, \qquad\qquad \qquad S _0= x_0, 
  \end{aligned}
      \right.
 \end{eqnarray*}  
where $\phi(x)=(x-K)^+$ and $f(t,x,y,z) = -(ry+\theta z).$
One can point out that in the Black \& Scholes  pricing framework, the value of the replication portfolio follows a linear BSDE. 
The value at time $t$ of the stochastic process $(Y_t)_{0\leq t \leq T}$ corresponds to the value of the portfolio and $ Z_t $ is related to  the hedging strategy. Our example shows that  in a complete market the value of the replicating portfolio and the hedging portfolio are associated with the  solution of a linear BSDE.\\ 
We can evaluate explicitly the value of the wealth process $Y$ for a fixed time. In particular, at the time instance  $t=0$,  $Y_0=  \E  \left( e^{-rT}\phi(S_T) \exp{(-\theta W_T + \frac{1}{2}\theta^2 T)} \right).$
By evaluating the preceding expectation, we obtain the classical  Black-Scholes formula  
\begin{align*}
Y_0 = e^{-r T }(F_{T} \mathcal{N}(d_+) - K \mathcal{N}(d_-)) \quad \text{and} \quad Z_0 =\sigma\mathcal{N}(d_+) S_0 , 
 \end{align*}
where,  $\displaystyle F_{T}=S_0e^{rT}, d_{\pm}= \frac{\log(F_{T}/K) \pm \frac{1}{2}\sigma^2T}{\sigma \sqrt{T}}$ and $\mathcal{N}(x)=\frac{1}{\sqrt{2\pi}} \int_{-\infty }^x e^{\frac{-t^2}{2}}dt.$
The function  $\mathcal{N}$  denotes the cumulative distribution function of the standard normal distribution.   
The objective is to provide a numerical solution of the system   $(E_1) $. 
We will be interested  in  the initial value of the couple $(Y ,Z)$.  We   suppose that we have at our disposal the value of the forward process $S$ on the grids of the partition $\pi$.

\vspace{0.5cm}
\noindent In our numerical  simulation,  we have considered a finite dimensional system of   normalized orthogonal Laguerre  polynomials. We have fixed the number of the chosen basis functions to be constant at each step of the algorithm and evaluate the couple $(Y,Z)$  along the time period $[0,T]$. \\
 Let us consider the  unidimensional discrete-time approximation of  the equation $(E_1)$. We build  the partition  $\pi$ of  the interval $[0,T]$ defined as follows:  
$$\pi: 0= t_0 < ... < t_N= T,$$
$\,\Delta_i:=t_{i+1}-t_i$  and $\abs{\pi} := \Max\{\Delta_i\,; 0\leq i \leq N-1\}$.
We set the  following parameters
 \begin{itemize}
\item  $k$  is the  number of   basis functions,
 \item $M$ is the number of simulated paths of the Brownian motion,
 \item $N$ is the number of the discretization points on $\pi$.
 \end{itemize}
 As input values, we  define  the following parameters 
 \begin{center}
\rule{\linewidth}{.5pt}
$ T=1, \quad r=0.01, \quad S_0=100, \quad K=100,\quad   \mu=1\%, \quad \sigma=2\%$
 \rule{\linewidth}{.5pt} 
 \end{center}

\noindent In the case of the European Call option, the exact value of the solution  $Y,$  at the time point  $t_0$ is   $Y_0=1.3886$ (value of the European  Call option contract) and  the exact   value  for $Z$ at the time point  $t_0$  is $ Z_0=1.39$. 
 The following figure shows the log-representation of the  relative error curve  induced by  the numerical  estimation of the couple $( {Y}_0, {Z}_0)$.    
  Modulo the choice on $|\pi|$ and number of   basis functions $k$,    
 the error decreases significantly as we increase the number of simulations $M$.    
Unfortunately   in   both cases below, the estimator  of the  couple $(Y,Z)$ could be subjected  to some bias in some particular cases of  the variation of the number of the selected   basis functions.   The  error  curves on  the estimation of   $Z_0$  seem to be more volatile. This fact can be justified by the gradient operator in  the regression-later algorithm.   Another effect   is   the accumulation of the projection error associated with the orthogonal projection operator.   
 
\vspace{0.5cm} 
\noindent
In the case of a European Put  option, the same argument as above leads  to a similar conclusion regarding the graphic analysis  of the computation of the price.  In this case, the exact value of the corresponding forward-backward SDE    at the time point  $t_0$ is   $(Y_0,Z_0)=(0.39,-0.60 )$. The exact value of  the European  Put   contract is  $Y_0=0.39 $.   In the case of a European Put  option, we obtain the same convergence order. 
\begin{figure}[H] 
   \begin{center}
  \includegraphics[scale=0.42]{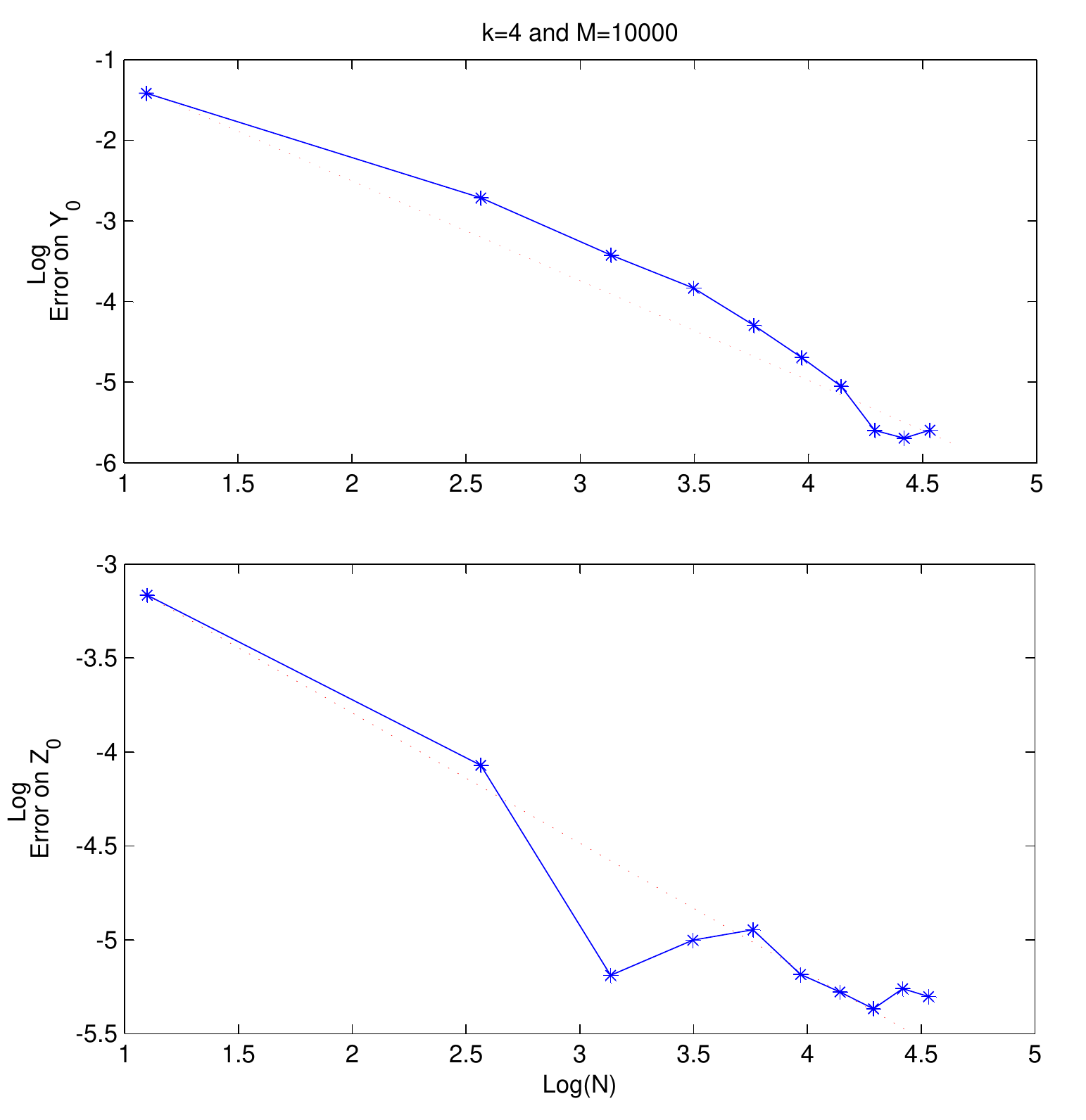}
      \caption{Log-Error  curve to estimate $(Y_0,Z_0)$, European Call case.}
	  \end{center}\label{fig_erreurCall}
	 \end{figure}
\begin{figure}[H]%
		 \begin{center}
 			\includegraphics[scale=0.50]{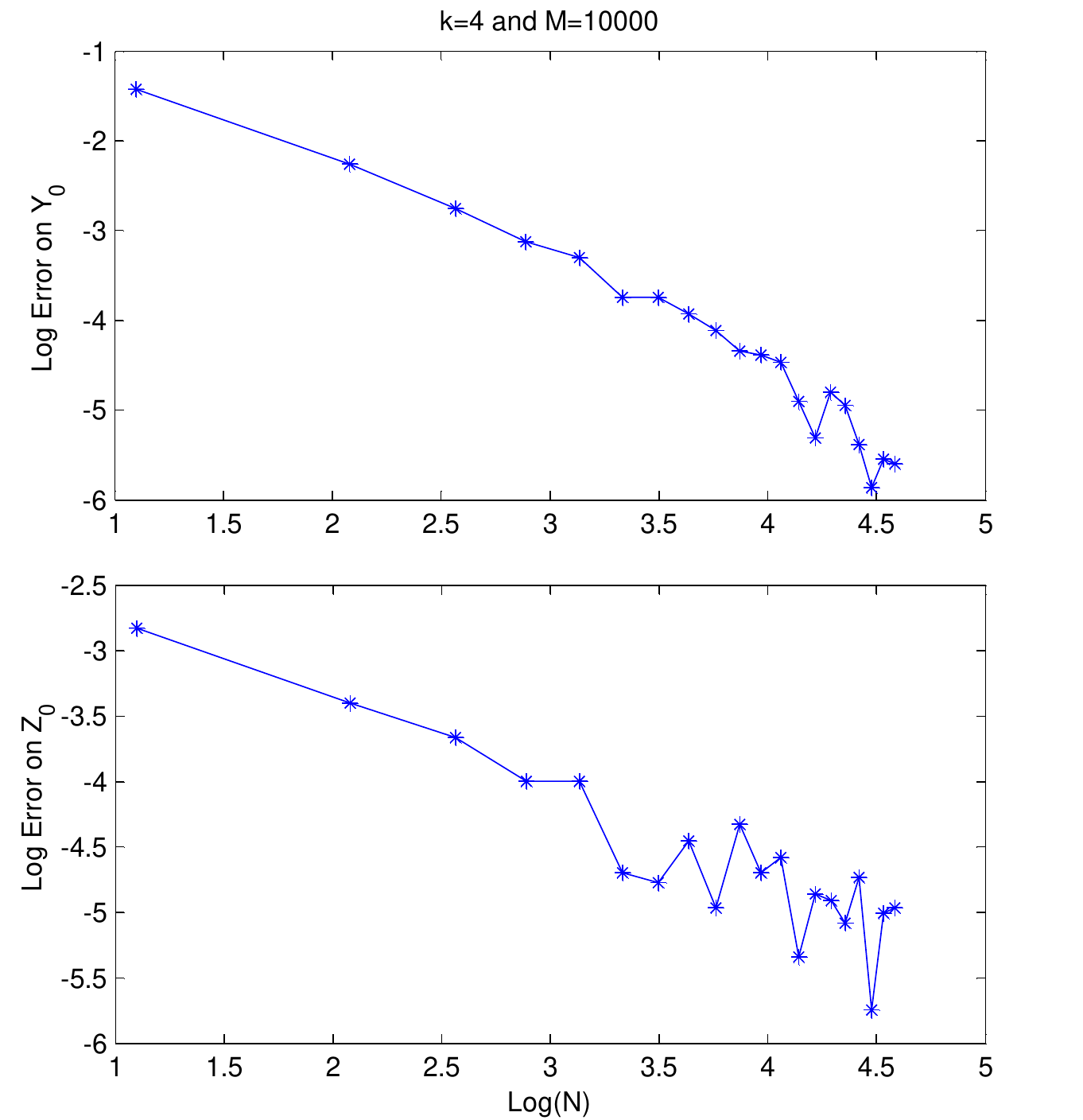}
            \caption{Log-Error  curve to estimate $(Y_0,Z_0)$, European Put.} 
						\end{center}
 \end{figure}
\noindent  On the  following graphic, we compare  the convergence   result  of the regression-later algorithm (in blue) with the standard  implicit Backward Euler-Maruyama scheme $(S.I)$  of  Section \ref{eulerMaruyama} in the particular  case of a Call option valuation.    The implicit Euler scheme  $(S.I)$ uses the  classical regression-now (cf. \cite{glasserman2004simulation})  technique to evaluate the couple $(Y_0,Z_0)$. It   is the customary approach represented in red. 
 
 \begin{figure}[H]
		 \begin{center}
  \includegraphics[scale=0.38]{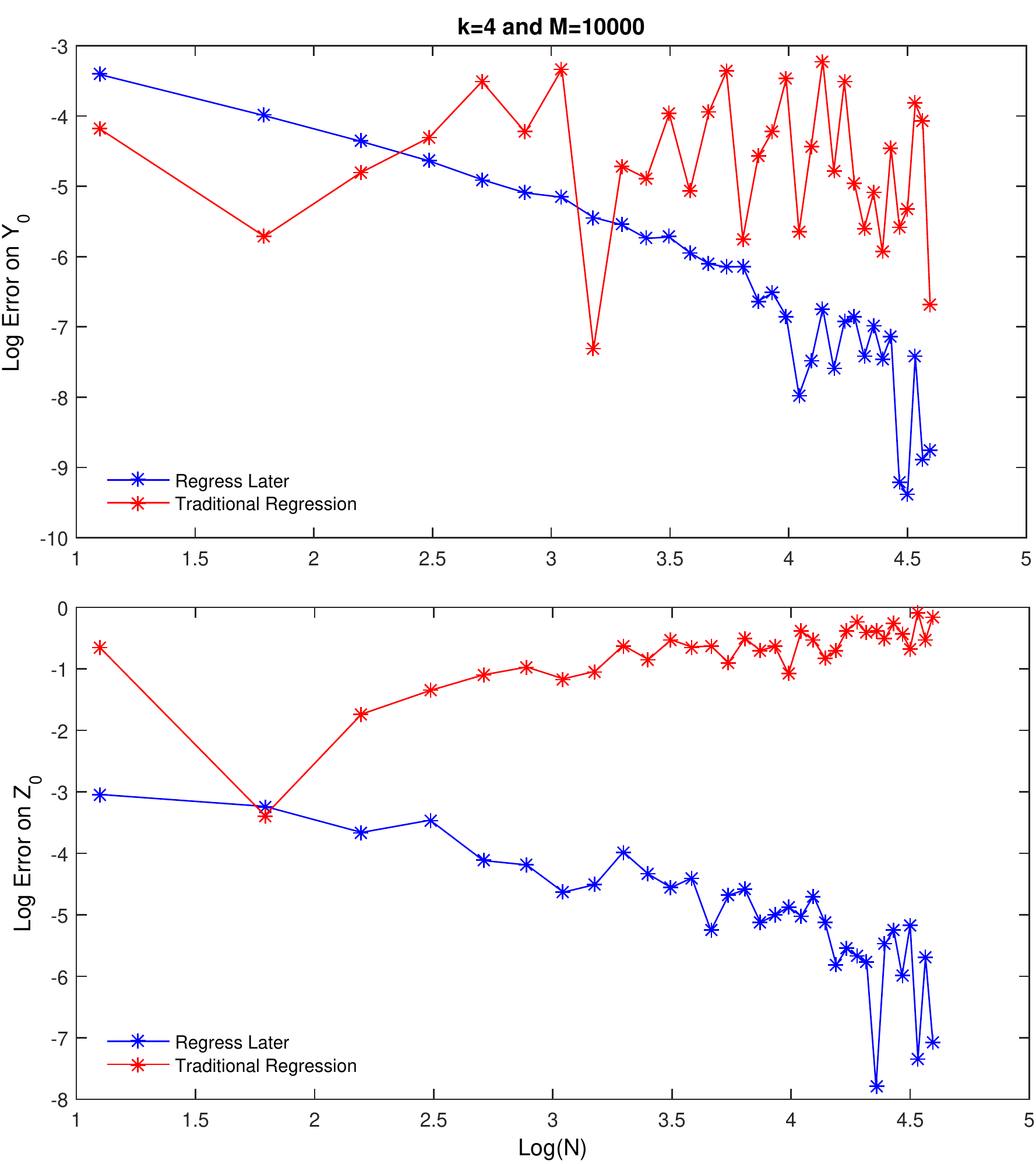}
        \caption{Comparison Log-Error  curve   (European Call).} 
	   	\end{center}
 \end{figure}
 \noindent The graphics of the above figure  shows   that the error curves are volatile when  we increase the number of time instances. The volatility effect seems to  be persistent   regarding the approximation of   $Z_0$ with the  scheme $(S.I)$.   
 In other words, the results  seem to be    more volatile with    the standard  implicit Backward Euler-Maruyama scheme $(S.I)$ in this particular case choice of  $M$ and $k$.   
   As an alternative approach, the regression-later approach  shows a stable  converge trend  and less volatile that the result of the scheme  scheme $(S.I)$.    The same remarks are  applied to the European Put case.  This  graphical  results  shows    that the regression-later approach: as an alternative approach, could offer several advantages  in comparison to  the  regression-now technique.  
 
 \subsection*{Application 2: Brownian Functional Case}
In this   example, the underlying process is assumed to be a standard Brownian motion  $W$ on the time interval  $[0 ,T]$. In other  words,  the forward  process is simply a  Brownian motion and the terminal condition is a functional of the Brownian motion  $W$.  We consider the BSDE 
  \begin{equation}
  \left\{ 
  \label{EDSR00_Ex2}
     \begin{aligned}
   	    -&dY_t=  f(t,W_t,Y_t,Z_t)dt - Z_tdW_t, \qquad 0\leq t< 1, \\
   	    &Y_1=\phi(W_1), 
   	    \end{aligned} 
  \right.
 \end{equation}
  where  the terminal function and the driver function are defined by    
 \begin{equation*}
     \left\{ 
     \begin{aligned}
     &\phi(x)= x\arctan(x)- \ln(\sqrt{ 1+x^2})\\
     & f(t,W_t	, Y_t, Z_t)= -  \frac{1}{2 (1+\tan^2(Z_t))}  . 
     \end{aligned} 
     \right.
 \end{equation*}
It is easy to check by   It\^o's formula  that, the solution of the above system is almost surely   
 %
   {$$(Y_t ,Z_t ) = (-  \frac{1}{2} \ln(1+W_t^2) +  W_t\arctan(W_t) , \arctan(W_t) ).$$ }
 By noting that   the function $x  \mapsto  \ln(x)$ satisfies    the linear growth  condition and the function $x  \mapsto  \arctan(x)$ is bounded,  the unique solution of \eqref{EDSR00_Ex2} satisfies    
 $$(Y_t ,Z_t)_{0\leq t\leq T}\in \mathcal{S}^2(\R ) \times \mathcal{H}^2(\R ) .$$
The exact value of the couple  $(Y,Z ) $  at the time point  $t_0$ is    $( {Y}  _0, {Z}  _0)=(0,0)$.
  The figure   below  shows  the empirical logarithm of the absolute  error  induced by  the numerical  estimation of the couple $( {Y}  _0, {Z}  _0)$.  
%
 %
 \begin{figure}[h!]
						\begin{center}
 				  \includegraphics[scale=0.50]{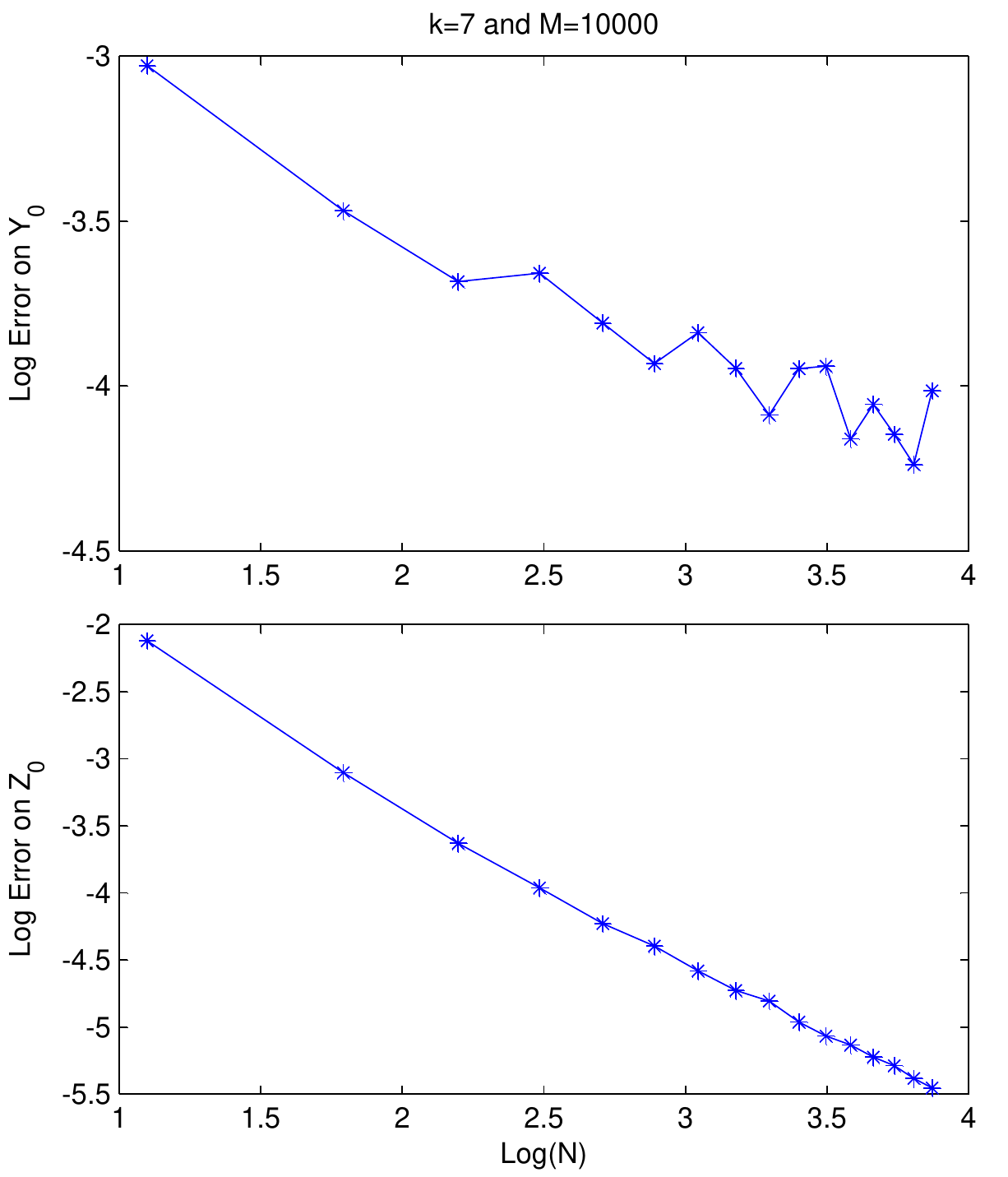}
                 \caption{Log-Error  curve to estimate  $(Y_0,Z_0)$}
						\end{center}
						\label{erreurTangente}
 \end{figure}
\noindent   Modulo the choice of $|\pi|$ and the   number of the basis function $k$,  the graphics show  a stable convergence result. This leads us to the same conclusion as above regarding the estimation  of the couple $ (Y_0 ,Z_0)$.       
Nevertheless the estimation of the initial value  $ Y_0$  is more stable and quicker  than the estimation of the initial value  $ Z_0$ in the first example. 
The convergence order could be also  accelerated  by   two-step schemes or   the Runge-Kutta methods (see e.g. \cite{chassagneux2014runge}, \cite{butcher2005runge}, \cite{ascher1997implicit}).
%
%
  \newpage
\section{Conclusion}
\setcounter{equation}{0} \setcounter{Assumption}{0}
\setcounter{Theorem}{0} \setcounter{Proposition}{0}
\setcounter{Corollary}{0} \setcounter{Lemma}{0}
\setcounter{Definition}{0} \setcounter{Remark}{0}
We have discussed a new numerical scheme for backward stochastic differential equations (BSDEs). 
 The  scheme is based on the regression-later approach. 
 In the first part of our work, we introduced the theory of BSDEs, gave  some general background on their studies and reviewed the classical backward Euler-Maruyama scheme.  In the next step, we   described our    regression-later algorithm  in detail  and  derived a  convergence result of the scheme. 
Finally, we   provided two numerical experiments to illustrate the performance of the regression-later  algorithm: the first in the context of option pricing  and  the second in  the case  where the terminal condition is a functional of a Brownian motion. 
Modulo a  suitable choice of the number of discretization points and the number of  basis functions,  our numerical results  show a stable convergence  regarding the estimation of   the  solution  $(Y,Z)$.
 In many numerical algorithms for solving BSDEs,  one of the
difficulties is to solve  a dynamic programming problem which involves often the computation of conditional expectations at each step across the time interval.
We remark that in  many alternative algorithms, the  numerical  computation of $Z$ is more challenging than the computation of the process $Y$, leading to potential numerical instabilities especially in higher dimensions. It is interesting to note that  our algorithm circumvents this difficulty,
 by obtaining the numerical approximation of the $Z$ process directly from the approximation of $Y$  and the basis functions. Our numerical results look highly promising, but more future researches are needed particularly regarding the global analysis of the error on the estimation of   $(Y, Z)$.  
 \newpage

\section{Appendix}
\setcounter{equation}{0} \setcounter{Assumption}{0}
\setcounter{Theorem}{0} \setcounter{Proposition}{0}
\setcounter{Corollary}{0} \setcounter{Lemma}{0}
\setcounter{Definition}{0} \setcounter{Remark}{0}
\begin{Lemma}\label{art1_Lem_YoungInequality}
 For any  constant $\alpha >0$  and for any $a,b\in \R$, 
  \begin{eqnarray}\label{ineq_YoungInequality}
 (a+b)^2 \leq  (1+\alpha) a^2 + (1+ \frac{1}{\alpha}) b^2.  
 \end{eqnarray} 
 \end{Lemma}

 \begin{proof}
 The result is a direct consequence of Young's inequality. 
 \end{proof}
\noindent Let us now recall the classical discrete Gronwall Lemma    (see, e.g. \cite{zhang2004NumSchemeBSDE} or \cite{memin2008convergence}) . 
\begin{Lemma}[\textbf{Gronwall Inequality A}]\label{art1_Gronwall1}
Let us consider  the partition \\
$$\pi: 0= t_0 < ... < t_N= T$$ of the interval $[0,T]$ and let  $\Delta_i$  
 be its mesh. 
We also consider the families  
$(a_k)_{0\leq k\leq N},(b_k)_{0\leq k\leq N},$ assumed  to be non-negative such that 
 for some positive constant   $\gamma>0$  we have: 
 \begin{eqnarray*}
 a_{k-1} \leq (1+\gamma \Delta_i)a_k+b_{k},\quad k=1,\dots,N   .
 \end{eqnarray*} 
Then,  
$$  \quad   \max_{0\leq i \leq N } a_i  \leq   e^{\gamma T} ( a_N + \sum_{i=1}^{N}  b_i ).$$
 \end{Lemma} 
\begin{Lemma}[\textbf{Gronwall Inequality B}]\label{art1_Gronwall2}
 Let $y,b, a : [0,T] \mapsto \R $ be three continuous functions such that,   $b$ is non-negative and
 $$   y(t) \leq a(t) +  \int_0^t b(s)y(s)ds, \qquad 0 \leq t \leq T.$$ 
  Then,  
 $$  y(t) \leq a(t) + \int_0^t a(s) b(s) \exp  \bigg(  \int_s^t b(u)du \bigg)  ds, \qquad 0 \leq t \leq T.     $$ 
In addition, if the function $a$ is non-decreasing, then
 $$  y(t) \leq a(t)  \exp  \bigg(  \int_0^t b(s)ds \bigg), \qquad 0 \leq t \leq T.     $$ 
\end{Lemma}
 \newpage
\bibliography{arXiv_Regress_Later} 
 \bibliographystyle{alpha}  %
\end{document}